\documentclass[10pt]{article}
\usepackage[T1]{fontenc}
\usepackage[utf8]{inputenc}
\usepackage{amsmath,amsfonts,amsthm,mathrsfs,amssymb}
\usepackage{cite}
\usepackage{multirow}
\usepackage{graphicx,float}
\usepackage{subfig}
\usepackage{placeins}
\usepackage{xcolor}
\usepackage{indentfirst}
\usepackage{enumitem}
\usepackage{caption}
\usepackage{float}
\usepackage{bm}
\usepackage{soul}%删除线

\usepackage[colorlinks,
            linkcolor=blue,
            anchorcolor=blue,
            citecolor=blue]{hyperref} %引用跳转
\numberwithin{equation}{section}
\newcommand{\R}{{\mathbb R}}

\newcommand{\be}{\begin{eqnarray}}
\newcommand{\en}{\end{eqnarray}}
\newcommand{\ben}{\begin{eqnarray*}}
\newcommand{\enn}{\end{eqnarray*}}

\newcommand{\bs}{\boldsymbol}
\newcommand{\pa}{\partial}

\newcommand{\real}{{\rm Re\,}}

\newcommand{\e}{\mathrm{e}}
\newcommand{\inc}{\mathrm{inc}}

\newcommand{\sct}{\mathrm{sct}}
\newcommand{\trn}{\mathrm{trn}}
\newcommand{\modi}{\mathrm{mod}}

\newtheorem{theorem}{Theorem}[section]

\newtheorem{remark}[theorem]{Remark}

\definecolor{rot}{rgb}{1.000,0.000,0.000}
\definecolor{rot1}{rgb}{0.000,0.000,0.000}
\begin{document}
\renewcommand{\theequation}{\arabic{section}.\arabic{equation}}
\renewcommand{\figurename}{Figure}
\captionsetup[figure]{labelformat=simple}
\renewcommand{\tablename}{Table}
\captionsetup[table]{labelformat=simple}

\begin{titlepage}
\title{A Uniformly High-Accuracy PML-BIE Method for Scattering by Periodic Arrays of Obstacles: The 2D Case}

\author{Yan Tan\thanks{School of Mathematical Sciences, University of Electronic Science and Technology of China, Chengdu, Sichuan 611731, China. Email: {\tt yantan@std.uestc.edu.cn}}\;,
Carlos P\'erez-Arancibia\thanks{Department of Applied Mathematics, University of Twente, Enschede, The Netherlands. Email: {\tt c.a.perezarancibia@utwente.nl}}\;,
Tao Yin\thanks{State Key Laboratory of Mathematical Sciences and Institute of Computational Mathematics and Scientific/Engineering Computing, Academy of Mathematics and Systems Science, Chinese Academy of Sciences, Beijing 100190, China. Email:{\tt yintao@lsec.cc.ac.cn}}}

\date{}
\end{titlepage}
\maketitle

\begin{abstract}
This paper presents a novel frequency-robust perfectly matched layer (PML) boundary integral equation (BIE) method for solving two-dimensional electromagnetic scattering problems involving periodic arrays of obstacles. In periodic scattering problems, standard BIE formulations based on the quasi-periodic Green's function require the evaluation of lattice sums or challenging Sommerfeld-type integrals, which diverge at Rayleigh--Wood (RW) anomalies. An alternative is to use BIE formulations based on the Helmholtz free-space Green's function, but these are defined on unbounded unit-cell boundaries and therefore require suitable truncation strategies, such as the Windowed Green Function (WGF) method. Although such approaches avoid the use of expensive quasi-periodic Green's functions, they also suffer from breakdowns at RW anomalies unless an appropriate mode correction is incorporated. Similarly, the direct application of PML-BIE techniques to periodic structures experiences comparable difficulties near RW anomalies due to the destruction of exponential convergence near RW anomalies for fixed PML parameters. To overcome this challenge, we propose a modified PML-BIE method that combines the PML technique with a finite-mode correction, ensuring both high accuracy and robustness at and around RW-anomalies. Convergence of the PML-truncated boundary integral operators is proved and several numerical examples are presented to validate the efficiency and performance of the proposed method.

{\bf Key words.} Scattering problem, boundary integral equation, Rayleigh--Wood anomaly, perfectly matched layer

% {\bf AMS subject classifications.} 31A10, 35C15, 78M15
\end{abstract}

\section{Introduction}
\label{sec:1}
Scattering problems in periodic or bi-periodic structures have many applications in micro-optics, including photonic crystal modeling, signal processing, and antenna design. This paper is devoted to proposing a new frequency-robust solver for the numerical solution of two-dimensional problem of electromagnetic scattering by periodic arrays of infinite penetrable obstacles with plane incidence by utilizing the boundary integral equation (BIE) method \cite{BH09,BH10,BD14,BF17,SFFP23} whose success lies in part on several distinct advantages, for example, inherent dimensionality reduction, requiring discretization only of the scattering surface rather than the surrounding volume; and avoidance of numerical dispersion errors while maintaining compatibility with high-order discretization schemes.

% A variety of numerical approaches have been proposed to tackle this important problem \cite{BR93,LP18,NP18,MNP80,BL22,AA15,MJ08}. 

Numerical methods that directly discretize the physical domain, such as the finite element method (FEM) \cite{CW03,BL22}, are widely used, but they usually require a large number of degrees of freedom to achieve high accuracy due to the need to introduce artificial boundary condition (ABC) for domain truncation and discretize the volumetric domain. In particular, construction and analysis of the classical ABCs, for example, the transparent boundary condition in terms of Dirichlet-to-Neumann (DtN) map and perfectly matched layer (PML) \cite{B94}, for solving the periodic structure scattering problems with plane incidence relies on the quasi-periodicity as well as the related Rayleigh expansion of the scattered field. Owing to the pioneer work \cite{LLQ18} and subsequent related contributions \cite{LXYZ23,BLYZ24,GL22,YHLR22,LZZ25,BFP24}, the idea of PML truncation in conjunction with the BIE (PML-BIE) method has shown its potentiality and efficiency for solving the scattering problems with infinite surface such as the layered-medium problem, periodic structure problem with local defects and water wave problem. 

Application of the PML-BIE method to the current periodic structure problem under consideration, which can be reduced to an unbounded unit-cell region by imposing an artificial quasi-periodic boundary condition, is straightforward (see section~\ref{sec:2.2}) and the resulting BIE only involves the integral operators with complex stretched free-space fundamental solution (FFS) which avoids the use of problematic quasi-periodic Green's function (QGF). In particular, the QGF can be represented as a Rayleigh expansion or a phased lattice sum of a free-space fundamental solution. The resulting BIE is only defined on the unbounded unit-cell boundaries and its numerical evaluation, however, suffers from slow convergence near or at the Rayleigh–Wood (RW) anomalies since the kernel itself becomes singular or fails to exist at the RW anomalies and one or more diffraction orders become grazing. A number of acceleration and regularization strategies have been developed to ensure efficiency and robustness near or at RW-anomaly configurations, including Ewald-type summation techniques for periodic lattice sums~\cite{L98,L10,ASSL13,CWJ07}, shifted Green's function methods~\cite{M08,BD14,BF17,BSTV16,BSTV17}, method of using hybrid spatial/spectral Green function representations and the Woodbury–-Sherman–-Morrison formula~\cite{F20,BF20}, among others. 
When quasi-periodicity of the fields is broken---for instance, in the case of a periodic structure with a local perturbation or illuminated by a point source---additional modifications are required; see \cite{LSZ26} for recent progress in this direction.

For periodic structure problems, ensuring uniform exponential convergence of the PML truncation---which generally takes the form $\mathcal{O}\!\left(\frac{d_{\min}}{e^{cd_{\min}}-1}\right)$, where $c>0$ is a constant depending on the PML parameters and $d_{\min}$ denotes the minimal distance between the frequency and the RW anomalies \cite{R07,W02}---requires the PML parameters to be chosen inversely proportional to $d_{\min}$. Consequently, for fixed PML parameters, the exponential convergence deteriorates as $d_{\min}\rightarrow 0$, since $\frac{d_{\min}}{e^{cd_{\min}}-1}\rightarrow \frac{1}{c}$ in this limit. As a result, a direct application of the PML-BIE method suffers from degraded convergence at frequencies near or at RW-anomaly configurations. This convergence breakdown has a counterpart in windowing-based truncations of the BIE in terms of the FFS \cite{SFFP23}: while superalgebraic convergence is achievable away from RW anomalies as the support of the window function grows, accuracy degrades severely near and at RW anomalies, because the naive windowing approximation fails to account for the radiation condition.

To properly enforce the radiation condition, i.e., the Rayleigh expansion satisfied by the scattered field, a modified windowed BIE utilizing mode-based corrections has been proposed in \cite{SFFP23} to provide a robust and superalgebraically convergent solver throughout the entire frequency spectrum, including at and around the challenging RW-anomaly configurations. Inspired by this idea, this work proposes a novel frequency-robust PML-BIE solver, with a different truncation mechanism, convergence behavior, and analysis, for the problem of scattering by a periodic array of obstacles. In particular, compared with the windowing truncation strategy of~\cite{SFFP23}, the incorporation of PML truncation allows for exponential convergence, achieving higher accuracy under the same truncation thickness, which can essentially reduce the number of degrees of freedom in practical computations. Near RW anomalies, the standard PML approximation also struggles to properly enforce the radiation condition at infinity. For both propagating and evanescent modes, the proved exponentially decaying properties of the tail term of PML-truncated boundary integral operators (BIOs)---whose dependence on the PML parameters, frequency, and RW anomalies is made explicit---indicate that an ill-conditioned system for the coefficients of non-propagating modes is introduced into the PML-BIE, persisting and polluting the numerical solution. To overcome this limitation, we develop a modified PML-BIE incorporating a finite-modes correction, ensuring both high accuracy and robustness across a broad frequency range, including at RW anomalies. The resulting method exhibits exponential convergence as the PML thickness increases, with PML parameter selection that is independent of the RW anomalies.
 
The paper is structured as follows. In Section~\ref{sec:2}, we introduce the model problem and the corresponding BIE formulation. Section~\ref{sec:3} is devoted to the derivation of the truncated BIE formulation using the PML technique, along with an analysis of the convergence of the PML-truncated tail term. In Section~\ref{sec:4}, we introduce a modified PML-BIE approach based on a finite-mode correction, designed to address the numerical challenges occurring at and around RW anomalies. Finally, in Section~\ref{sec:5}, we present several numerical examples to illustrate the performance of the proposed method.

\section{Preliminaries}
\label{sec:2}
\subsection{Model problem}
\label{sec:2.1}

Let $\{\Omega_2^{(i)}\}_{i\in\mathbb{Z}}$ denote a set of infinite periodic arrays of bounded, penetrable obstacles of class $C^2$ in two dimensions with period $\Lambda>0$, such that for any $i\in\mathbb{Z}$ and any $x=(x_1,x_2)\in\Omega_2^{(i)}$, we have $(x_1+n\Lambda,x_2)\in\Omega_2^{(i+n)}$ for all $n\in\mathbb{Z}$. Define $D_2=\bigcup_{i\in\mathbb{Z}} \Omega_2^{(i)}$ and $D_1=\mathbb{R}^2 \setminus \overline{D_2}$. Assume that each domain $D_j$ ($j=1,2$) is filled with a homogeneous medium characterized by constant electric permittivity $\epsilon_j>0$ and constant magnetic permeability $\mu_j>0$. Denoting the angular frequency by $\omega>0$, the wavenumber in $D_j$ is $k_j = \omega \sqrt{\epsilon_j \mu_j}$. In this work, we consider both TE and TM polarizations and study the scattering of an incident plane wave impinging on the periodic array $D_2$. The study of the corresponding three-dimensional problems is left for future work.

Let $\theta^{\inc}\in(-\frac{\pi}{2},\frac{\pi}{2})$ denote the incident angle. The incident plane wave is then given by
$$
u^{\inc}(x)=\e^{\mathrm{i}(\alpha x_1-\beta x_2)},\quad \alpha=k_1\sin\theta^{\inc},\quad \beta=k_1\cos\theta^{\inc},
$$
where $\alpha$ is the \emph{Bloch phase} (Bloch wavenumber) associated with the periodicity of the structure. Clearly, the incident field satisfies the Helmholtz equation
\ben
\Delta u^{\inc} + k_1^2 u^{\inc} = 0 \quad \text{in } \mathbb{R}^2.
\enn

Both TE and TM polarization scattering problems are modeled by the Helmholtz equation for the total field $u$:
\be\label{Helmholtz equation for total field}
\Delta u+k_j^2u=0\quad\text{in}\quad D_j,\quad j=1,2,
\en
subject to the transmission boundary conditions
\be\label{transmission boundary conditions}
\gamma_{D,\pa D_2}^+ u=\gamma_{D,\pa D_2}^- u\quad\text{and}\quad \gamma_{N,\pa D_2}^+ u=\eta\gamma_{N,\pa D_2}^- u\quad \text{on}\quad \pa D_2,
\en
where $\eta:=\mu_1/\mu_2$ in TE polarization and $\eta:=\epsilon_1/\epsilon_2$ in TM polarization. The Dirichlet and Neumann traces are defined as
\be\begin{aligned}
(\gamma_{D,\Gamma}^\pm u)(x)=&~\lim_{h\to 0^+,z=x\pm h\nu_x}u(z) \quad \text{and}\\ \quad (\gamma_{N,\Gamma}^\pm u)(x)=&~\lim_{h\to 0^+,z=x\pm h\nu_x}\nabla u(z)\cdot\nu_x,
\end{aligned}\en
for a given curve $\Gamma$ with unit normal vector $\nu_x$ at $x\in\Gamma$. 

The total field $u$ in \eqref{Helmholtz equation for total field} can be decomposed as
\ben
u=\begin{cases}
   u^{\sct}+u^{\inc},& \text{in}\quad D_1,\\
   u^{\trn},& \text{in}\quad D_2,
\end{cases}
\enn
in terms of the incident ($u^{\inc}$), transmitted ($u^\trn$), and scattered ($u^{\sct}$) fields. Then, on $\Gamma_1$, the scattered field $u^{\sct}$ and the transmitted field $u^\trn$ satisfy
\be\label{transmission boundary conditions1}\begin{aligned}
\gamma_{D,\pa D_2}^+ u^{\sct}=&~\gamma_{D,\pa D_2}^- u^\trn- f\quad\text{and}\\
\gamma_{N,\pa D_2}^+ u^{\sct}=&~\eta\gamma_{N,\pa D_2}^- u^\trn-g\quad \text{on}\quad \pa D_2,
\end{aligned}\en
where we have let $f:=\gamma_{D,\pa D_2}^+ u^{\inc}$ and $g:=\gamma_{N,\pa D_2}^+ u^{\inc}$. 

It follows from the periodicity of the scatterer $D_2$ and the plane wave incidence that we can restrict our discussion to quasi-periodic fields:
\be\label{quasi-periodicity condition}
u(x_1+\Lambda,x_2)=\zeta u(x_1,x_2),\quad \zeta:=\e^{\textbf{i}\alpha \Lambda},\quad (x_1,x_2)\in\mathbb{R}^2,
\en
so that the domain's problem ($\R^2$) can be reduced to a unit-cell domain 
\begin{equation}\label{eq:unit_cell_domian}
U=\{(x_1,x_2)\in\mathbb{R}^2:\hat{f}(t)<x_1<\hat{f}(t)+\Lambda,x_2=t,t\in\mathbb{R} \}
\end{equation}
which lies in the region between the two infinite parallel simple curves
\be
\label{parameterization}
\begin{aligned}
\Gamma_2&:=\{ x\in\mathbb{R}^2:x={\bf r}_2(t)=(\hat{f}(t),t),t\in\mathbb{R} \}\quad\text{and}\\
\Gamma_3&:=\{ x\in\mathbb{R}^2:x={\bf r}_3(t)=(\hat{f}(t)+\Lambda,t),t\in\mathbb{R} \},
\end{aligned}
\en
that are parameterized in terms of a smooth function $\hat{f}:\R\to \R$. 

For simplicity, we assume that $\overline{\Omega_2} \subset U$, with $\Omega_2 = \Omega_2^{(0)}$, and denote $\Gamma_1 = \partial \Omega_2:=\{x\in\R:x=\bs r_1(t),t\in I_{2\pi}:=[0,2\pi)\}$ (see Figure~\ref{geometrysetting}). For convenience in the analysis, we also set
$\hat{f}(t) = -\frac{\Lambda}{2}$ for all  $|t| > t_0$,
where $t_0$ is taken sufficiently large.

\begin{figure}[htbp]
\centering
\includegraphics[scale=1]{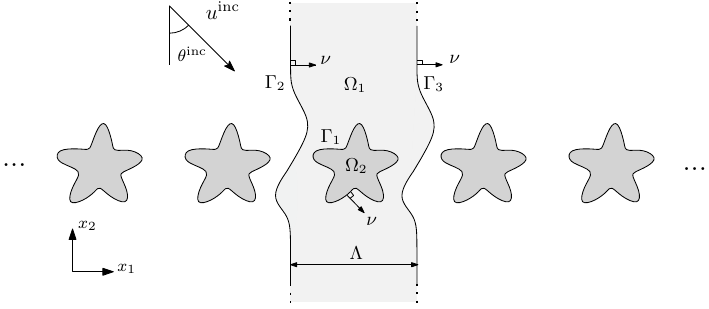}
\caption{Domain of the periodic scattering problem. The figure illustrates the geometry of the  problem and introduces the notation used.}
\label{geometrysetting}
\end{figure}

In addition, to ensure the well-posedness of the problem under consideration, the scattered field $u^{\sct}$ is assumed to satisfy the so-called upward/downward Rayleigh expansion:
\be
\label{RayleighExp}
u^{\sct}(x)=\sum_{n\in\mathbb{Z}} B_n^\pm \e^{\textbf{i}(\alpha_n x_1\pm \beta_n x_2)}\quad\mbox{for}\quad \pm x_2> \pm h^\pm,
\en
where $h^+:=\sup_{(x_1,x_2)\in D_2}x_2$, $h^-:=\inf_{(x_1,x_2)\in D_2}x_2$,  
\ben
\alpha_n=\alpha+\frac{2\pi n}{\Lambda} \quad \text{and}\quad\beta_n=\begin{cases}
   \sqrt{k_1^2-\alpha_n^2}, & \text{for} \quad |\alpha_n|\leq k_1,\\
   \mathbf{i}\sqrt{k_2^2-\alpha_n^2},& \text{for} \quad |\alpha_n|>k_1.
\end{cases}
\enn
Moreover, as is well known,  the following energy-balance relation~\cite{MNP80,BF17}
\be\label{energy balance}
2\real(B_0^+)+\sum_{n\in\mathcal{P}}\frac{\beta_n}{\beta}\left\{|B_n^-|^2+|B_n^+|^2\right\}=0,
\en 
is satisfied.

As discussed bellow, however, the incorporation of the radiation condition, as expressed by the Rayleigh expansion (\ref{RayleighExp}), is not straightforward within our BIE framework. To address this, we introduce an equivalent formulation of the radiation condition by projecting the scattered field onto the non-propagating modes (non-physical solutions) as follows, see also \cite{SFFP23}.

We define the following subsets of integers:
\be
\begin{aligned}
\mathcal{P}&:=\{n\in\mathbb{Z}:|\alpha_n|<k_1\},\label{setP}\\
\mathcal{Q}&:=\{n\in\mathbb{Z}:|\alpha_n|>k_1\},\label{setQ}\\
\mathcal{N}&:=\{n\in\mathbb{Z}:|\alpha_n|=k_1\},\label{setN}\\
\mathcal{C}_\delta&:=\{n\in\mathbb{Z}:0\leq|\beta_n|\leq\delta\}.\label{setC}
\end{aligned}
\en
In particular, the set  $\mathcal{N}=\mathcal{N}(k_1,\alpha,L)=\{n\in\mathbb{Z}:(\alpha+2\pi n/\Lambda)^2=k_1^2\}=\{n\in\mathbb{Z}:\beta_n=0\}$ corresponds to the RW-anomaly configuration. It is easy to show that a quasi-periodic solution $u^{\sct}$ to the homogeneous Helmholtz equation $\Delta u^{\sct} + k_1^2 u^{\sct} = 0$ in $D_1$  admits the general series expansion
\be\label{weak RayleighExp}
\hspace{1cm}u^{\sct}=\sum_{n\in\mathcal{P}\cup\mathcal{Q}}\{B_n^\pm u_n^\pm  +C_n^\pm u_n^\mp\}+\sum_{n\in\mathcal{N}}\{ B_n^\pm u_n+C_n^\pm v_n \}\text{ for } \pm x_2> \pm h^\pm,
\en
which consists of both physical and non-physical solutions associated with the coefficients $B_n^\pm$ and $C_n^\pm$, respectively, where
\be\label{upgoing and downgoing waves}
u_n^\pm (x_1,x_2):=\e^{\textbf{i}\alpha_nx_1\pm \textbf{i}\beta_nx_2},\quad \text{and} \quad u_n^-(x_1,x_2):=\e^{\textbf{i}\alpha_nx_1-\textbf{i}\beta_nx_2}
\en
are upgoing and downgoing waves for $n\in\mathcal{P}$ and evanescent waves for $n\in\mathcal{Q}$, respectively. In addition,
\be
u_n(x_1,x_2)=\e^{\textbf{i}\alpha_n x_1},\quad n\in\mathcal{N}
\en
is a plane wave propagating parallel to the array along the $x_1$-axis and
\be\label{horizontally propagating modes}
v_n(x_1,x_2):=x_2\e^{\textbf{i}\alpha_n x_1},\quad n\in\mathcal{N},
\en
is a degenerate solution. The fact that $u^{\sct}$ is radiative and bounded in the sense of (\ref{RayleighExp}) implies that $C_n^\pm=0$ for all $n\in\mathbb{Z}$. Therefore, by projecting $u^{\sct}(x_1,\pm h)$ and $\pa_{x_2}u^{\sct}(x_1,\pm h)$ for $h>\max\{h^+,-h^-\}$ onto $\e^{\textbf{i}\alpha_n x}$, we obtain the relation
\ben\begin{aligned}
C_n^\pm=\frac{1}{\Lambda}\int_{-\frac{\Lambda}{2}}^{\frac{\Lambda}{2}}\left( \pa_{x_2}u^{\sct}(x_1,\pm h)\mp \textbf{i}\beta_n u^{\sct}(x_1,\pm h) \right)\e^{-\textbf{i}\alpha_n x_1}\mathrm{d}x_1\cdot\\\begin{cases}
   \frac{\mp 1}{2\mathbf{i}\beta_n}\e^{\textbf{i}\beta_n h}&\text{if}\quad n\in\mathcal{P}\cup\mathcal{Q},\\
   1,&\text{if}\quad n\in\mathcal{N}.
\end{cases}  
\end{aligned}\enn
Therefore, the radiation condition (\ref{RayleighExp}) can be equivalently enforced by requiring $u^{\sct}$ to satisfy (\ref{weak RayleighExp}) and
\be\label{weaker radiation condition}
\frac{1}{\Lambda}\int_{-\frac{\Lambda}{2}}^{\frac{\Lambda}{2}}\left( \pa_{x_2}u^{\sct}(x_1,\pm h)\mp \textbf{i}\beta_n u^{\sct}(x_1,\pm h) \right)\e^{-\textbf{i}\alpha_n x_1}\mathrm{d}x_1=0,\quad n\in\mathbb{Z}.
\en
This weaker form (\ref{weaker radiation condition}) of the radiation condition is more amenable to implementation within our BIE formulation.

\subsection{PML boundary integral equation formulation}
\label{sec:2.2}
Following~\cite{SFFP23}, in this section we present the solution representation and the corresponding PML-BIE formulation for the scattering problem introduced in the previous section. Unlike the windowing function used in~\cite{SFFP23}, this work employs a PML complex-stretching approach, which allows the BIE solver to achieve exponential convergence rather than super-algebraic convergence. 

As is well known, a PML along the $x_2$-axis can be implemented by extending the (real) spatial coordinates into the complex plane using a complex stretching function of the form
\be
\widetilde{x}_2(x_2)=x_2+\textbf{i}\int_0^{x_2}\sigma(t)\mathrm{d}t,\quad x_2\in\R,
\en
where the real-valued function $\sigma$ is assumed to satisfy
\ben
\sigma(t)=\sigma(-t),\quad \sigma(t)=0~\text{for}~|t|\leq H,\quad \sigma(t)> 0~\text{for}~|t|> H,
\enn
for some $H>\max\{h^+,-h^-\}>0$, which we referred to as the half-width of the physical domain. The function $\sigma$ can be chosen as a power function~\cite{CW03}, or as another type of function; see, e.g.,~\cite{LLQ18,LXYZ23,BLYZ24}. We here employ:
\be\label{complex stretching function}
\sigma(x_2)=\begin{cases}
   S\left(\frac{x_2-H}{T}\right)^P,&H\leq x_2\leq H+T,\\
   S,&x_2>H+T,\\
   0,&-H\leq x_2\leq H,\\
   \sigma(-x_2),&x_2\leq -H,
\end{cases}
\en
where $T>0$ denotes the PML thickness, $P\geq 2$ is a positive integer, and $S>0$ is the PML scaling parameter.

We now introduce the following notations for the analytically extended total, scattered, and incident fields:
\begin{align}
\widetilde{u}(x) &= u(\widetilde{x}), & 
\widetilde{u}^{\sct}(x) &= u^{\sct}(\widetilde{x}), & 
\widetilde{u}^{\inc}(x) &= u^{\inc}(\widetilde{x}),
\end{align}
respectively. It can be readily shown, by applying the chain rule, that $\widetilde{u}^{\sct}$ satisfies the modified Helmholtz equation 
\be
\label{modified Helmholtz scattered field}
\nabla\cdot(\mathbb{A}^{-1}\nabla \widetilde{u}^{\sct})+k_1^2J\widetilde{u}^{\sct}=0 \quad\mbox{in}\quad\Omega_1,
\en
where $\mathbb{A}=J^{-1}\mathbb{B}\mathbb{B}^\top$ and $J=\alpha(x_2)$ with $\alpha(x_2)=1+\textbf{i}\sigma(x_2)$ and $\mathbb{B}=\text{diag}\{1,\alpha(x_2)\}$. 

Note that the transmitted field $u^\trn$ is not affected by the complex stretching and since $\widetilde{u}^{\sct}(x) = u^{\sct}(x)$ for $|x_2| \le H$, the complex stretching does not alter the transmission conditions~\eqref{transmission boundary conditions1} on $\Gamma_1$. The quasi-periodicity condition~\eqref{quasi-periodicity condition} also remains unchanged, as the complex stretching is applied only along the $x_2$-axis. As a result, in view of~\eqref{RayleighExp}, $\widetilde{u}^{\sct}$ admits the complex-stretched Rayleigh expansion
\be
\label{PMLRayleighExp}
\widetilde{u}^{\sct}(x)=\sum_{n\in\mathbb{Z}} B_n^\pm \e^{\textbf{i}(\alpha_n x_1\pm \beta_n \widetilde x_2)}\quad\mbox{for}\quad \pm x_2> \pm h^\pm.
\en

As shown in~\cite[Theorem~2.8]{LS01}, the fundamental solution to equation 
$\nabla\cdot(\mathbb{A}^{-1}\nabla v) + k_j^2 J v = 0$ is given by
\be
\widetilde{\Phi}(k_j,x,y)=\Phi(k_j,\widetilde{x},\widetilde{y})=\frac{\textbf{i}}{4}H_0^{(1)}(k_j\rho(\widetilde{x},\widetilde{y})),\quad j=1,2.
\en
where $\widetilde{x}=(x_1,\widetilde{x}_2)$ and $\widetilde{y}=(y_1,\widetilde{y}_2)$; the complex distance function $\rho$ is given by
\be \label{eq:complex_distance_function}
\rho(\widetilde{x},\widetilde{y})=\left[(x_1-y_1)^2+(\widetilde{x}_2-\widetilde{y}_2)^2\right]^{1/2},
\en
and $z^{1/2}$ is chosen to be the branch of $\sqrt{z}$ with nonnegative real part for $z\in\mathbb{C}\setminus(-\infty,0]$ such that $\text{arg}(z^{1/2})\in(-\frac{\pi}{2},\frac{\pi}{2}]$. Here, $H_0^{(1)}$ is  the Hankel function of the first kind with order zero.

%
%We begin by letting $\Phi(k_j,x,y)$ denote the free-space Green's function of the Helmholtz equation with wavenumber $k_j>0$, which is given by 
%\ben
%\Phi(k_j,x,y):=\frac{\textbf{i}}{4}H_0^{(1)}(k_j|x-y|),\quad j=1,2,
%\enn
%where $H_0^{(1)}$ is  the Hankel function of the first kind with order zero.  

With these ingredients at hand, we are now in a position to present the PML–BIE formulation of the problem, which follows directly from~\cite[secs. 3--5]{SFFP23}, with the proviso that the integral operators are defined using the free-space Green’s function evaluated using the complexified distance~\eqref{eq:complex_distance_function}. For the sake of brevity, we present below only the main results of these derivations while introducing the necessary notation.

For a given smooth density $\varphi:\Gamma_i\to\mathbb{C}$, we define the single-layer potential $\widetilde{\mathscr{S}}_j^{\ \Gamma_i}$ and the double-layer potential $\widetilde{\mathscr{D}}_j^{\Gamma_i}$,  for $j=1,2$ and $i=1,2,3$, as follows:
\begin{equation}\label{PML-transformed hypersingular LPs}
\begin{aligned}
\widetilde{\mathscr{S}}_j^{\ \Gamma_i}[\varphi](x)&:=\int_{\Gamma_i}\widetilde\Phi(k_j,x,y)\varphi(y)\mathrm{d}s_y,\quad x\in\Omega_j,\\
\widetilde{\mathscr{D}}_j^{\Gamma_i}[\varphi](x)&:=\int_{\Gamma_i}\widetilde{\pa}_{\nu_y}\widetilde\Phi(k_j,x,y)\varphi(y)\mathrm{d}s_y,\quad x\in\Omega_j.
\end{aligned}
\end{equation}
where $\widetilde{\pa}_\nu v=(\mathbb{A}^{-1}\nu)\cdot\nabla v$. 

We then denote by $\widetilde S_j^{\Gamma_l,\Gamma_i}$, $\widetilde D_j^{\Gamma_l,\Gamma_i}$, $\widetilde K_j^{\Gamma_l,\Gamma_i}$ and $\widetilde N_j^{\Gamma_l,\Gamma_i}$  the associated single-layer, double-layer, adjoint double-layer and hypersingular BIOs,  which are respectively defined by
\begin{equation}
\begin{aligned}
\widetilde S_j^{\Gamma_l,\Gamma_i}[\varphi](x)&:=\int_{\Gamma_i}\widetilde\Phi(k_j,x,y)\varphi(y)\mathrm{d}s_y,\quad x\in\Gamma_l,\\
\widetilde D_j^{\Gamma_l,\Gamma_i}[\varphi](x)&:=\int_{\Gamma_i}\widetilde{\pa}_{\nu_y}\widetilde\Phi(k_j,x,y)\varphi(y) \mathrm{d}s_y,\quad x\in\Gamma_l,\\
\widetilde K_j^{\Gamma_l,\Gamma_i}[\varphi](x)&:=\int_{\Gamma_i}\widetilde{\pa}_{\nu_x}\widetilde\Phi(k_j,x,y)\varphi(y)\mathrm{d}s_y,\quad x\in\Gamma_l,\\
\widetilde N_j^{\Gamma_l,\Gamma_i}[\varphi](x)&:=\text{f.p.}\int_{\Gamma_i}\widetilde{\pa}_{\nu_x}\widetilde{\pa}_{\nu_y}\widetilde\Phi(k_j,x,y)\varphi(y)\mathrm{d}s_y,\quad x\in\Gamma_l\label{PML-transformed hypersingular BIO}.
\end{aligned}
\end{equation}
where ``f.p." in the definition of the hypersingular operator refers to the fact that the integral should be understood as a Hadamard finite part integral. With $\delta_{l,i}$ denoting the Kronecker delta, we have the following well-known jump relations
\begin{equation}
\begin{aligned}
\left(\gamma_{D,\Gamma_l}^\pm \widetilde{\mathscr{S}}_j^{\ \Gamma_i}\right)\varphi=\widetilde S_j^{\Gamma_l,\Gamma_i}\varphi,&\qquad \left(\widetilde\gamma_{N,\Gamma_l}^\pm \widetilde{\mathscr{S}}_j^{\ \Gamma_i}\right)\varphi=\mp\delta_{l,i}\frac{\varphi}{2}+\widetilde D_j^{\Gamma_l,\Gamma_i}\varphi,\\
\left(\gamma_{D,\Gamma_l}^\pm \widetilde{\mathscr{D}}_j^{\Gamma_i}\right)\varphi=\pm\delta_{l,i}\frac{\varphi}{2}+\widetilde K_j^{\Gamma_l,\Gamma_i}\varphi,&\qquad \left(\widetilde\gamma_{N,\Gamma_l}^\pm \widetilde{\mathscr{D}}_j^{\Gamma_i}\right)\varphi=\widetilde N_j^{\Gamma_l,\Gamma_i}\varphi,
\end{aligned}
\end{equation}
where the Neumann trace is defined as 
\ben
 (\widetilde{\gamma}_{N,\Gamma_l}^\pm \widetilde{u})(x)=\lim_{h\to 0^+}(\mathbb{A}^{-1}(x)\nu_x)\cdot\nabla \widetilde{u}(x\pm h\nu_x),\qquad x\in\Gamma_l.
\enn
In particular, $(\widetilde{\gamma}_{N,\Gamma_1}^\pm\widetilde{u})(x)=(\gamma_{N,\Gamma_1}^\pm u)(x)$.

Using the curve parameterization given in~\eqref{parameterization}, for $i=2,3$, the layer potentials $\mathscr{F}_j^{\Gamma_i}$, with $\mathscr{F}\in\{\widetilde{\mathscr{S}},\widetilde{\mathscr{D}}\}$, and the BIOs $F_j^{\Gamma_l,\Gamma_i}$, with $F\in\{\widetilde S,\widetilde D,\widetilde K, \widetilde N\}$, can be written as
\be
&&\mathscr{F}_j^{\Gamma_i}[\phi](x):=\int_{\R} \mathcal{H}(k_j,x,{\bf r}_i(\tau))\phi(\tau)|{\bf r}_i'(\tau)|\mathrm{d}\tau,\quad x\in\Omega_j,\label{Para-BIO1}\\
&&F_j^{\Gamma_l,\Gamma_i}[\phi](x):=\int_{\R}\mathcal{H}(k_j,x,{\bf r}_i(\tau))\phi(\tau)|{\bf r}_i'(\tau)|\mathrm{d}\tau,\quad x\in\Gamma_l,\label{parameter-LP}
\en
where generic kernel $\mathcal{H}\in\{\widetilde\Phi,\widetilde{\pa}_{\nu_y}\widetilde\Phi,\widetilde{\pa}_{\nu_x}\widetilde\Phi,\widetilde{\pa}_{\nu_x}\widetilde{\pa}_{\nu_y}\widetilde\Phi\}$ (resp. $\mathcal{H}\in\{\widetilde\Phi,\widetilde{\pa}_{\nu_y}\widetilde\Phi\}$) corresponds to generic operator $F\in\{\widetilde S,\widetilde D,\widetilde K,\widetilde N\}$ (resp. $\mathscr{F}\in\{\widetilde{\mathscr{S}},\widetilde{\mathscr{D}}\}$) and we define $\phi(\tau)=\varphi\circ\bs r_i$. Furthermore, when the target point $x$ also lies on a curve $\Gamma_l$, $l,i=2,3$, we can, slightly abusing the notation, define the following generic parameterized BIOs:
\begin{equation}
F_j^{\Gamma_l,\Gamma_i}[\phi](t):=\int_{\R}\mathcal{H}(k_j,{\bf r}_l(t),{\bf r}_i(\tau))\phi(\tau)|{\bf r}_i'(\tau)|\mathrm{d}\tau,\quad t\in\R.\label{Para-BIO2}
\end{equation}
The same definitions apply for the cases $i = 1$ and $l = 1$, with $\mathbb{R}$ replaced by the bounded interval $I_{2\pi}$.

Following the derivations in~\cite{SFFP23}, {with the windowing procedure replaced by the PML complex stretching and truncation,} an application of Green's representation theorem yields that the fields within the unit-cell domain $U$ can be represented in integral form as
\begin{subequations}\begin{eqnarray}\label{scattered field in exterior domain}
\widetilde u^{\sct}&=&\widetilde{\mathscr{D}}_1^{\Gamma_1}[\phi_1]-\eta\widetilde{\mathscr{S}}_1^{\ \Gamma_1}[\phi_2]+\widetilde{\mathscr{D}}_1^{\Gamma_2}[\phi_3]-\widetilde{\mathscr{S}}_1^{\ \Gamma_2}[\phi_4]\\
&&\quad-\zeta\big\{\widetilde{\mathscr{D}}_1^{\Gamma_3}[\phi_3](x)-\widetilde{\mathscr{S}}_1^{\ \Gamma_3}[\phi_4]\big\}\quad\quad \mbox{in}\quad\Omega_1,
\nonumber\\
\label{scattered field in interior domain}
\widetilde u^\trn&=&-\widetilde{\mathscr{D}}_2^{\Gamma_1}[\phi_1]+\widetilde{\mathscr{S}}_2^{\ \Gamma_1}[\phi_2]\qquad\qquad\qquad \mbox{in}\quad\Omega_2,
\end{eqnarray}\label{eq:C_potentials}\end{subequations}
in terms of the parametrized traces
\ben\label{density functions}
\begin{aligned}
\phi_1&:=\left(\gamma_{D,\Gamma_1}^-u^\trn\right)\circ\bs r_1:I_{2\pi}\to\mathbb{C}, &\phi_2&:=\left(\widetilde{\gamma}_{N,\Gamma_1}^-u^\trn\right)\circ\bs r_1:I_{2\pi}\to\mathbb{C},\\
\phi_3&:=\left(\gamma_{D,\Gamma_2}^+ \widetilde u^{\sct}\right)\circ\bs r_2:\R\to\mathbb{C}, &\phi_4&:=\left(\widetilde \gamma_{N,\Gamma_2}^+ \widetilde  u^{\sct}\right)\circ \bs r_2:\R\to\mathbb{C}.
\end{aligned}
\enn
% where the Neumann trace is defined as 
% \ben
%  (\widetilde{\gamma}_{N,\Gamma_2}^+ \widetilde{u})(x)=\lim_{h\to 0^+}(\mathbb{A}^{-1}(x)\nu_x)\cdot\nabla \widetilde{u}(x+ h\nu_x),\qquad x\in\Gamma_2.
% \enn
%We also write $\widetilde{\phi}_q(t)= \widetilde{\phi}_q({\bf r}_2(t))$ for $q=3,4, t\in\R$.

In view of both the transmission conditions~\eqref{transmission boundary conditions1} and the quasi-periodicity condition~\eqref{quasi-periodicity condition}, taking the traces of the potentials in~\eqref{scattered field in exterior domain} and~\eqref{scattered field in interior domain} leads to the following BIE system:
\begin{equation}
\label{BIE-original}
(\mathbb{E}+\mathbb{T})[\bs \phi]=\bs\phi^{\inc}\quad \mbox{on}\quad (I_{2\pi})^2\times(\R)^2,
\end{equation}
where $\boldsymbol{\phi}=[\phi_1,\phi_2,\phi_3,\phi_4]^\top$, $\bs \phi^{\inc}=[f\circ\boldsymbol r_1,g\circ \boldsymbol r_1,0,0]^\top$ and the diagonal matrix $\mathbb{E}$ and block matrix-valued BIO $\mathbb{T}$  are given by
\ben
\mathbb{E}=\mathrm{diag}\left\{ 1,\frac{1+\eta}{2},\zeta,\zeta \right\}\quad\text{and}\quad\mathbb{T}=\left[\mathbb{T}_{p,q}\right]_{p,q=1}^4,
\enn
respectively. Note that we consider the vector-valued function $\bs\phi$ as a column vector. In detail, the elements of the block matrix-valued BIO $\mathbb{T}$ are
\ben
\begin{aligned}
&\mathbb{T}_{1,1}:=\widetilde D_2^{\Gamma_1,\Gamma_1}-\widetilde D_1^{\Gamma_1,\Gamma_1},&\mathbb{T}_{1,2}&:=\eta \widetilde S_1^{\Gamma_1,\Gamma_1}-\widetilde S_{2}^{\Gamma_1,\Gamma_1},\\
&\mathbb{T}_{1,3}:=\zeta \widetilde D_1^{\Gamma_1,\Gamma_3}-\widetilde D_1^{\Gamma_1,\Gamma_2},&\mathbb{T}_{1,4}&:=\widetilde S_1^{\Gamma_1,\Gamma_2}-\zeta  \widetilde S_1^{\Gamma_1,\Gamma_3},\\
&\mathbb{T}_{2,1}:=\widetilde N_2^{\Gamma_1,\Gamma_1}-\widetilde N_1^{\Gamma_1,\Gamma_1},&\mathbb{T}_{2,2}&:=\eta \widetilde K_1^{\Gamma_1,\Gamma_1}- \widetilde K_{2}^{\Gamma_1,\Gamma_1},\\
&\mathbb{T}_{2,3}:=\zeta \widetilde N_1^{\Gamma_1,\Gamma_3}-\widetilde N_1^{\Gamma_1,\Gamma_2},&\mathbb{T}_{2,4}&:=\widetilde K_1^{\Gamma_1,\Gamma_2}-\zeta \widetilde K_1^{\Gamma_1,\Gamma_3},\\
&\mathbb{T}_{3,1}:=-\zeta \widetilde D_1^{\Gamma_2,\Gamma_1}-\widetilde D_1^{\Gamma_3,\Gamma_1},&\mathbb{T}_{3,2}&:=\eta(\zeta  \widetilde S_1^{\Gamma_2,\Gamma_1}+\widetilde S_1^{\Gamma_3,\Gamma_1}),\\
&\mathbb{T}_{3,3}:=\zeta^2 \widetilde D_1^{\Gamma_2,\Gamma_3}-\widetilde D_1^{\Gamma_3,\Gamma_2},&\mathbb{T}_{3,4}&:=\widetilde S_1^{\Gamma_3,\Gamma_2}-\zeta^2 \widetilde S_1^{\Gamma_2,\Gamma_3},\\
&\mathbb{T}_{4,1}:=-\zeta \widetilde N_1^{\Gamma_2,\Gamma_1}-\widetilde N_1^{\Gamma_3,\Gamma_1},&\mathbb{T}_{4,2}&:=\eta(\zeta \widetilde K_1^{\Gamma_2,\Gamma_1}+\widetilde K_1^{\Gamma_3,\Gamma_1}),\\
&\mathbb{T}_{4,3}:=\zeta^2 \widetilde N_1^{\Gamma_2,\Gamma_3}-\widetilde N_1^{\Gamma_3,\Gamma_2},&\mathbb{T}_{4,4}&:=\widetilde K_1^{\Gamma_3,\Gamma_2}-\zeta^2\widetilde K_1^{\Gamma_2,\Gamma_3}.
\end{aligned}
\enn
{The resulting block operator $\mathbb{T}$ takes the same algebraic structure as the one derived in~\cite{SFFP23}; the only difference is that the present formulation involves PML complex-stretched layer potentials instead of the windowed layer potentials.}

Since the solution representations \eqref{scattered field in exterior domain}–\eqref{scattered field in interior domain} and the BIEs \eqref{BIE-original} involve integration over the unbounded curves $\Gamma_j$, $j=2,3$, a suitable procedure is required to truncate the integrals over these curves, even though the integrands decay exponentially as a consequence of the complex stretching. 

\section{Truncation of the PML-BIE system and convergence analysis}
\label{sec:3}

In this section we present the direct truncation of the PML-transformed BIEs defined on $\Gamma_j,j=1,2,3$ as well as the corresponding convergence study near the RW-anomalies.

\subsection{Truncated PML-BIE}
\label{sec:3.1}

\begin{figure}[htbp]
   \centering
   \includegraphics[scale=1]{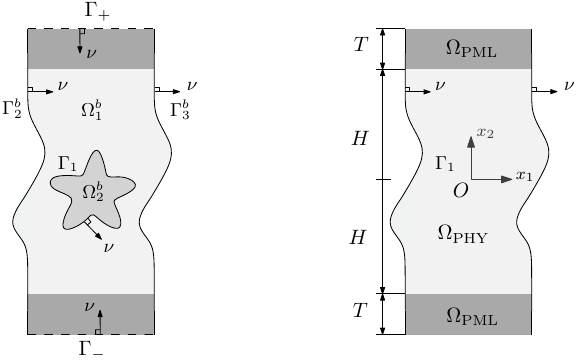}
%   \vspace{-0.1in}
   \caption{Description of the PML truncated domain.}
   \label{fig:PML truncated domain}
\end{figure}

We begin this section by introducing additional notation required to formulate the BIE~\eqref{BIE-original} on the truncated domain.
We follow the convention that the superscript ``$b$" denotes bounded sets (after appropriate truncation), and the symbols ``$+$" and ``$-$" indicate whether the corresponding set lies above or below the horizontal lines $x_2 = +(H+T)$ and $x_2 = -(H+T)$, respectively.

More precisely, as shown in Figure~\ref{fig:PML truncated domain}, let
$
\Omega_{\mathrm{PHY}} = U_H := U \cap \{|x_2| < H\} \subset \mathbb{R}^2
$
denote the truncated physical domain, where $U$ is the unit cell defined in~\eqref{eq:unit_cell_domian}.
We further define
$
U_{H,T} := U \cap \{|x_2| < H+T\},
$
which includes both the upper and lower PML regions. In addition, we introduce the exterior domain
$
U_{H,T}^c := U_{H,T}^+ \cup U_{H,T}^-,
$
where the top ($+$) and bottom ($-$) regions are given by
$
U_{H,T}^\pm := U \cap \{\pm x_2 > H+T\}.
$ With these definitions, we let:
\ben
&\Omega_{\mathrm{PML}} := U_{H,T} \setminus \overline{U_H},
\quad \Omega_i^b := U_{H,T} \cap \Omega_i,
\quad \Omega_i^\pm := \Omega_i \cap U_{H,T}^\pm,
\quad i = 1,2,\\
&\Gamma_q^b := \Gamma_q \cap U_{H,T},
\quad \Gamma_q^\pm := \Gamma_q \cap U_{H,T}^\pm,
\quad q = 1,2,3,\\
&\Gamma_\pm := \left\{ (x_1,x_2)\in\mathbb{R}^2 : -\frac{\Lambda}{2} < x_1 < \frac{\Lambda}{2},\ x_2 = \pm(H+T) \right\}.
\enn
Note in particular that
$$
\Omega_2^b = \Omega_2, \quad \Omega_2^\pm = \emptyset, \quad
\Gamma_1^b = \Gamma_1, \quad \Gamma_1^\pm = \emptyset.
$$

We next define the layer potentials
$\widetilde{\mathscr{F}}_j^{\Gamma_i^b}$ (resp. $\widetilde{\mathscr{F}}_j^{\Gamma_i^\pm}$)
and the boundary integral operators (BIOs)
$\widetilde{F}_j^{\Gamma_l^b,\Gamma_i^b}$ (resp. $\widetilde{F}_j^{\Gamma_l^b,\Gamma_i^\pm}$),
where $\widetilde{\mathscr{F}} \in {\widetilde{\mathscr{S}}, \widetilde{\mathscr{D}}}$ and
$\widetilde{F} \in {\widetilde{S}, \widetilde{D}, \widetilde{K}, \widetilde{N}}$. These operators are defined analogously to those in \eqref{PML-transformed hypersingular LPs}, with the boundaries $\Gamma_i$ replaced by $\Gamma_i^b$ (resp. $\Gamma_i^\pm$) and $\Gamma_l$ replaced by $\Gamma_l^b$, for $l,i = 1,2,3$. With these notations, the PML-transformed boundary integral system \eqref{BIE-original}, when restricted to
$(I_{2\pi})^2 \times (I_{H+T})^2$
(with $I_{H+T} = (-H-T, H+T)$ and $I_{2\pi} = [0, 2\pi)$),
can be equivalently written as:
\be\label{equivalent PML-BIE}
(\mathbb{E}+\mathbb{\widetilde{T}}^b)[\widetilde{\bs \phi}]={\bs \phi}^{\inc}-{\bs \psi},\quad\mbox{on}\quad (I_{2\pi})^2 \times (I_{H+T})^2,
\en 
where ${\bs \psi}=[\psi_1, \psi_2, \psi_3, \psi_4]^\top:=\mathbb{\widetilde{T}}^+[\widetilde{\bs \phi}]+\mathbb{\widetilde{T}}^-[\widetilde{\bs \phi}]$ and the elements of the block matrix-valued BIO $\mathbb{\widetilde{T}}^b=\left[\widetilde{\mathbb{T}}^b_{p,q}\right]_{p,q=1}^4$ and $\mathbb{\widetilde{T}}^\pm=\left[\widetilde{\mathbb{T}}^\pm_{p,q}\right]_{p,q=1}^4$ are given as follows:
\ben
\label{Tb}
\begin{aligned}
&\mathbb{\widetilde{T}}^b_{1,1}=\widetilde D_2^{\Gamma_1^b,\Gamma_1^b}-\widetilde D_1^{\Gamma_1^b,\Gamma_1^b},&\mathbb{\widetilde{T}}^b_{1,2}&=\eta \widetilde S_1^{\Gamma_1^b,\Gamma_1^b}-\widetilde S_{2}^{\Gamma_1^b,\Gamma_1^b},\\
&\mathbb{\widetilde{T}}^b_{1,3}=\zeta \widetilde D_1^{\Gamma_1^b,\Gamma_3^b}-\widetilde D_1^{\Gamma_1^b,\Gamma_2^b},&\mathbb{\widetilde{T}}^b_{1,4}&=\widetilde S_1^{\Gamma_1^b,\Gamma_2^b}-\zeta \widetilde S_1^{\Gamma_1^b,\Gamma_3^b},\\
&\mathbb{\widetilde{T}}^b_{2,1}=\widetilde N_2^{\Gamma_1^b,\Gamma_1^b}-\widetilde N_1^{\Gamma_1^b,\Gamma_1^b},&\mathbb{\widetilde{T}}^b_{2,2}&=\eta \widetilde K_1^{\Gamma_1^b,\Gamma_1^b}-\widetilde K_{2}^{\Gamma_1^b,\Gamma_1^b},\\
&\mathbb{\widetilde{T}}^b_{2,3}=\zeta \widetilde N_1^{\Gamma_1^b,\Gamma_3^b}-\widetilde N_1^{\Gamma_1^b,\Gamma_2^b},&\mathbb{\widetilde{T}}^b_{2,4}&=\widetilde K_1^{\Gamma_1^b,\Gamma_2^b}-\zeta \widetilde K_1^{\Gamma_1^b,\Gamma_3^b},\\
&\mathbb{\widetilde{T}}^b_{3,1}=-\zeta \widetilde D_1^{\Gamma_2^b,\Gamma_1^b}-\widetilde D_1^{\Gamma_3^b,\Gamma_1^b},&\mathbb{\widetilde{T}}^b_{3,2}&=\eta(\zeta \widetilde S_1^{\Gamma_2^b,\Gamma_1^b}+\widetilde S_1^{\Gamma_3^b,\Gamma_1^b}),\\
&\mathbb{\widetilde{T}}^b_{3,3}=\zeta^2\widetilde D_1^{\Gamma_2^b,\Gamma_3^b}-\widetilde D_1^{\Gamma_3^b,\Gamma_2^b},&\mathbb{\widetilde{T}}^b_{3,4}&=\widetilde S_1^{\Gamma_3^b,\Gamma_2^b}-\zeta^2\widetilde S_1^{\Gamma_2^b,\Gamma_3^b},\\
&\mathbb{\widetilde{T}}^b_{4,1}=-\zeta \widetilde N_1^{\Gamma_2^b,\Gamma_1^b}-\widetilde N_1^{\Gamma_3^b,\Gamma_1^b},&\mathbb{\widetilde{T}}^b_{4,2}&=\eta(\zeta \widetilde K_1^{\Gamma_2^b,\Gamma_1^b}+\widetilde K_1^{\Gamma_3^b,\Gamma_1^b}),\\
&\mathbb{\widetilde{T}}^b_{4,3}=\zeta^2\widetilde N_1^{\Gamma_2^b,\Gamma_3^b}-\widetilde N_1^{\Gamma_3^b,\Gamma_2^b},&\mathbb{\widetilde{T}}^b_{4,4}&=\widetilde K_1^{\Gamma_3^b,\Gamma_2^b}-\zeta^2\widetilde K_1^{\Gamma_2^b,\Gamma_3^b},
\end{aligned}
\enn
and, noting that $\Gamma_1^\pm=\emptyset$, we have 
\ben
\label{Tpm}
\begin{aligned}
&\mathbb{\widetilde{T}}^\pm_{1,1}=0,&\mathbb{\widetilde{T}}^\pm_{1,2}&=0,\\
&\mathbb{\widetilde{T}}^\pm_{1,3}=\zeta \widetilde D_1^{\Gamma_1^b,\Gamma_3^\pm}-\widetilde D_1^{\Gamma_1^b,\Gamma_2^\pm},&\mathbb{\widetilde{T}}^\pm_{1,4}&=\widetilde S_1^{\Gamma_1^b,\Gamma_2^\pm}-\zeta \widetilde S_1^{\Gamma_1^b,\Gamma_3^\pm},\\
&\mathbb{\widetilde{T}}^\pm_{2,1}=0,&\mathbb{\widetilde{T}}^\pm_{2,2}&=0,\\
&\mathbb{\widetilde{T}}^\pm_{2,3}=\zeta \widetilde N_1^{\Gamma_1^b,\Gamma_3^\pm}-\widetilde N_1^{\Gamma_1^b,\Gamma_2^\pm},&\mathbb{\widetilde{T}}^\pm_{2,4}&=\widetilde K_1^{\Gamma_1^b,\Gamma_2^\pm}-\zeta \widetilde K_1^{\Gamma_1^b,\Gamma_3^\pm},\\
&\mathbb{\widetilde{T}}^\pm_{3,1}=0,&\mathbb{\widetilde{T}}^\pm_{3,2}&=0,\\
&\mathbb{\widetilde{T}}^\pm_{3,3}=\zeta^2\widetilde D_1^{\Gamma_2^b,\Gamma_3^\pm}-\widetilde D_1^{\Gamma_3^b,\Gamma_2^\pm},&\mathbb{\widetilde{T}}^\pm_{3,4}&=\widetilde S_1^{\Gamma_3^b,\Gamma_2^\pm}-\zeta^2\widetilde S_1^{\Gamma_2^b,\Gamma_3^\pm},\\
&\mathbb{\widetilde{T}}^\pm_{4,1}=0,&\mathbb{\widetilde{T}}^\pm_{4,2}&=0,\\
&\mathbb{\widetilde{T}}^\pm_{4,3}=\zeta^2\widetilde N_1^{\Gamma_2^b,\Gamma_3^\pm}-\widetilde N_1^{\Gamma_3^b,\Gamma_2^\pm},&\mathbb{\widetilde{T}}^\pm_{4,4}&=\widetilde K_1^{\Gamma_3^b,\Gamma_2^\pm}-\zeta^2\widetilde K_1^{\Gamma_2^b,\Gamma_3^\pm},
\end{aligned}
\enn

Our preliminary truncated PML–BIE method neglects the exponentially small term $\bm \psi$ and solves the following approximate system:
\begin{equation}\label{truncated PML-BIE}
\left(\mathbb{E}+\mathbb{\widetilde{T}}^b\right)[\widehat{\boldsymbol \phi}]={\bs \phi}^{\inc},\quad \mbox{on}\quad (I_{2\pi})^2 \times (I_{H+T})^2,
\end{equation}
where $\widehat {\boldsymbol\phi}=[\widehat{{\phi}}_1, \widehat{{\phi}}_2, \widehat{{\phi}}_3, \widehat{{\phi}}_4]^\top$. Once $\widehat{\boldsymbol\phi}$ is obtained, the approximate fields in $\Omega_1$ and $\Omega_2$ are retrieved via the truncated layer potentials~\eqref{eq:C_potentials}:
\begin{align}\label{truncated sol ext}
\widetilde{u}^{\sct}&\approx \widetilde{\mathscr{D}}_1^{\Gamma_1^b}[{\widehat\phi}_1]-\eta\widetilde{\mathscr{S}}_1^{\Gamma_1^b}[{\widehat\phi}_2]+\widetilde{\mathscr{D}}_1^{\Gamma_2^b}[\widehat{\phi}_3]-\widetilde{\mathscr{S}}_1^{\Gamma_2^b}[\widehat{\phi}_4]\\
&\quad-\zeta\big\{ \widetilde{\mathscr{D}}_1^{\Gamma_3^b}[\widehat{\phi}_3]-\widetilde{\mathscr{S}}_1^{\Gamma_3^b}[\widehat{\phi}_4] \big\}\ \quad\quad \mbox{in}\quad\Omega_1,\nonumber\\
\label{truncated sol int}
u^\trn&\approx -\widetilde{\mathscr{D}}_2^{\Gamma_1^b}[{\widehat\phi}_1]+\widetilde{\mathscr{S}}_2^{\Gamma_1^b}[{\widehat\phi}_2]\quad\qquad\qquad \mbox{in}\quad\Omega_2.
\end{align}

\begin{remark}
\label{pmlfail}
In connection to the original PML-BIE method for layered-medium scattering problems discussed in \cite{LXYZ23,BLYZ24}, our truncated PML-BIE can alternatively be derived from the transmission problem for the modified Helmholtz equation $\nabla\cdot(\mathbb{A}^{-1}\nabla \widetilde{u}^{\sct})+k_1^2J\widetilde{u}^{\sct}=0$ in $U_{H,T}$ by setting $\widetilde{u}^{\sct}= \widetilde{\pa}_\nu\widetilde{u}^{\sct}=0$ on $\Gamma_\pm$ in the Green representation formula for the scattered field, leveraging the exponential convergence of PML truncation \cite{CZ10,CZ17}. In contrast, the PML method in \cite{CW03} requires only the Dirichlet boundary condition $\widetilde{u}^{\sct}=0$ on $\Gamma_\pm$ and yields a PML truncation error of order $\mathcal{O}\left(\frac{\min_n\{|\beta_n|\}}{\e^{c\min_n\{|\beta_n|\}}-1}\right)$, where $c>0$ is a constant depending on the PML parameters. For fixed PML parameters, this exponential convergence degenerates as the frequency approaches the RW-anomalies (even when the wavenumber is not small), i.e., when $\min_n\{|\beta_n|\}\rightarrow 0$.
\end{remark}

\subsection{Decaying properties of the truncated term}\label{sec:3.2}
In this section we focus on studying whether the approximation $\boldsymbol\psi=0$ used in the derivation of the truncated BIE system~\eqref{truncated PML-BIE} is justified and whether we should expect exponential convergence of the PML solutions as the PML length $T$ increases.

For the sake of presentation simplicity and without loss of generality, in the remainder of this section we consider straight unit-cell vertical boundaries outside the region of interest, i.e., $\hat{f}(t)=-\frac{\Lambda}{2}$ for all $|t|>H+T$. Similar to the derivation of \eqref{weak RayleighExp}, the PML-transformed radiation condition \eqref{PMLRayleighExp} can be recast into a less direct form, which admits the following general PML-transformed series expansions
$$
\widetilde{u}^\sct=\sum_{n\in\mathcal{P}\cup\mathcal{Q}}\Big\{B_n^\pm\widetilde{u}_n^\pm +C_n^\pm\widetilde{u}_n^\mp\Big\}+\sum_{n\in\mathcal{N}}\Big\{ B_n^\pm\widetilde{u}_n+C_n^\pm\widetilde{v}_n \Big\}\quad\text{for}\quad \pm x_2>H+T,
$$
where
\begin{align*}
\widetilde{u}_n^+(x_1,x_2):=&~u_n^+(x_1,\widetilde{x}_2),\qquad& \widetilde{u}_n^-(x_1,x_2):=&~u_n^-(x_1,\widetilde{x}_2),\\
\widetilde{u}_n(x_1,x_2):=&~\e^{\textbf{i}\alpha_n x_1},\qquad\text{and}& \widetilde{v}_n(x_1,x_2):=&~v_n(x_1,\widetilde{x}_2)=\widetilde{x}_2\e^{\textbf{i}\alpha_n x_1}.
\end{align*}
The corresponding parametrized traces ${\phi}_3$ and ${\phi}_4$ on the unbounded exterior boundaries $\Gamma_2^\pm$, can then be expressed as
\begin{equation}\label{PML complement trace on Gamma_2}
\begin{aligned}
\widetilde{\phi}_3\left(\pm t\right)&=\sum_{n\in\mathcal{P}\cup\mathcal{Q}}\e^{-\mathbf{i}\alpha_n\frac{\Lambda}{2}}\left\{ B_n^\pm \e^{\pm\mathbf{i}\beta_n \tilde x_2(t)}+C_n^\pm \e^{\mp\mathbf{i}\beta_n\tilde x_2(t)} \right\}\\
&+\sum_{n\in\mathcal{N}}\e^{-\mathbf{i}\alpha_n\frac{\Lambda}{2}}\left\{B_n^\pm+C_n^\pm \tilde x_2(t)\right\},\\
\widetilde{\phi}_4(\pm t)&=\sum_{n\in\mathcal{P}\cup\mathcal{Q}}\mathbf{i}\alpha_n\e^{-\mathbf{i}\alpha_n\frac{\Lambda}{2}}\left\{ B_n^\pm \e^{\pm\mathbf{i}\beta_n\tilde x_2(t)}+C_n^\pm \e^{\mp\mathbf{i}\beta_n\tilde x_2(t)} \right\}\\
&+\sum_{n\in\mathcal{N}}\mathbf{i}\alpha_n\e^{-\mathbf{i}\alpha_n\frac{\Lambda}{2}}\left\{B_n^\pm+C_n^\pm \tilde x_2(t)\right\},
\end{aligned}
\end{equation}
for $|t|>H+T$. In particular, the upward/downward Rayleigh expansion requires that $C_n^\pm=0$, however, which is not explicitly incorporated into the truncated PML-BIE formulation (\ref{truncated PML-BIE}). To further simplify the notation, we introduce the following PML-transformed Dirichlet and Neumann traces of Rayleigh modes on $\Gamma_1^b$ and $\Gamma_2^b$:
\be\label{PML-transformed Dirichlet and Neumann traces of Rayleigh modes}
\begin{aligned}
\phi_{n,1}^\pm&:=(\gamma_{D,\Gamma_1^b}u_n^\pm)(x)=\e^{\mathbf{i}\alpha_nx_1\pm\mathbf{i}\beta_nx_2},\quad x\in\Gamma_1^b,\\ 
\phi_{n,2}^\pm&:=(\gamma_{N,\Gamma_1^b}u_n^\pm)(x)=\mathbf{i}\nu_x\cdot(\alpha_n,\pm\beta_n)\e^{\mathbf{i}\alpha_nx_1\pm\mathbf{i}\beta_nx_2},\quad x\in\Gamma_1^b,\\
\widetilde{\phi}_{n,3}^\pm&:=(\widetilde{\gamma}_{D,\Gamma_2^b}\widetilde{u}_n^\pm)(x)=\e^{-\mathbf{i}\alpha_n\frac{\Lambda}{2}\pm\mathbf{i}\beta_n\widetilde{x}_2},\quad x\in\Gamma_2^b,\\
\widetilde{\phi}_{n,4}^\pm&:=(\widetilde{\gamma}_{N,\Gamma_2^b}\widetilde{u}_n^\pm)(x)=\mathbf{i}(\mathbb{B}\mathbb{A}^{-1}\nu_x)\cdot(\alpha_n,\pm\beta_n)\e^{-\mathbf{i}\alpha_n\frac{\Lambda}{2}\pm\mathbf{i}\beta_n\widetilde{x}_2},\quad x\in\Gamma_2^b.
\end{aligned}
\en
Substituting \eqref{PML complement trace on Gamma_2} into $\bs \psi$ in \eqref{equivalent PML-BIE} directly yields
\ben
\psi_p=\psi_{p,B}+\psi_{p,C},\quad p=1,\ldots,4\quad \mbox{on}\quad (I_{2\pi})^2 \times (I_{H+T})^2,
\enn
where
\be\qquad\label{tail integral terms of B} 
\psi_{p,B}=\sum_{n\in\mathbb{Z}}B_n^+\left\{ \mathbb{\widetilde{T}}_{p,3}^+[\widetilde{\phi}_{n,3}^+]+\mathbb{\widetilde{T}}_{p,4}^+[\widetilde{\phi}_{n,4}^+] \right\}+\sum_{n\in\mathbb{Z}}B_n^-\left\{ \mathbb{\widetilde{T}}_{p,3}^-[\widetilde{\phi}_{n,3}^-]+\mathbb{\widetilde{T}}_{p,4}^-[\widetilde{\phi}_{n,4}^-] \right\},
\en
and
\be\label{tail integral terms of C}
\begin{aligned}
{\psi_{p,C}}&=\sum_{n\in\mathcal{P}\cup\mathcal{Q}}C_n^+\left\{ \mathbb{\widetilde{T}}_{p,3}^+[\widetilde{\phi}_{n,3}^-]+\mathbb{\widetilde{T}}_{p,4}^+[\widetilde{\phi}_{n,4}^-] \right\}+\sum_{n\in\mathcal{P}\cup\mathcal{Q}}C_n^-\left\{ \mathbb{\widetilde{T}}_{p,3}^-[\widetilde{\phi}_{n,3}^+]+\mathbb{\widetilde{T}}_{p,4}^-[\widetilde{\phi}_{n,4}^+] \right\}\\
&\quad+\sum_{n\in\mathcal{N}}\e^{-\mathbf{i}\alpha_n\frac{\Lambda}{2}}C_n^+\left\{ \mathbb{\widetilde{T}}_{p,3}^+[\widetilde{t}]+\mathbf{i}\alpha_n\mathbb{\widetilde{T}}_{p,4}^+[\widetilde{t}] \right\}+\sum_{n\in\mathcal{N}}\e^{-\mathbf{i}\alpha_n\frac{\Lambda}{2}}C_n^-\left\{ \mathbb{\widetilde{T}}_{p,3}^-[\widetilde{t}]+\mathbf{i}\alpha_n\mathbb{\widetilde{T}}_{p,4}^-[\widetilde{t}]\right\}.
\end{aligned}
\en

We then proceed to examine the terms $\psi_{p,B}$ and $\psi_{p,C}$ in \eqref{tail integral terms of B}-\eqref{tail integral terms of C}. For $\beta_n\in\mathbb{R}^+$ or $\beta_n\in\mathbf{i}\mathbb{R}^+$, that is $n\in\mathcal{P}\cup\mathcal{Q}$, we have the following theorem.

\begin{theorem}
\label{convergence1}
Let $x\in\Gamma_1$. As $T\to\infty$, the PML osicillatory integral operator $\widetilde{S}_1^{\Gamma_1,\Gamma_2^+}[\e^{\pm \mathbf{i}\beta_n\widetilde{y}_2}](x)$ decays as $\mathcal{O}(\e^{-(k_1\pm\beta_n)\frac{ST}{P+1}})$ for $n\in\mathcal{P}$ and as $\mathcal{O}(\e^{-k_1\frac{ST}{P+1}\mp|\beta_n|(H+T)})$ for $n\in\mathcal{Q}$.
\end{theorem}
\begin{proof}
Without loss of generality, we focus our analysis on integral operator $\widetilde{S}_1^{\Gamma_1,\Gamma_2^+}[\e^{\mathbf{i}\beta_n\widetilde{y}_2}](x)$. The case involving $\widetilde{S}_1^{\Gamma_1,\Gamma_2^+}[\e^{-\mathbf{i}\beta_n\widetilde{y}_2}](x)$ can be treated in a completely analogous manner, yielding similar results. Denote $I_T(x):=\widetilde{S}_1^{\Gamma_1,\Gamma_2^+}[\e^{\mathbf{i}\beta_n\widetilde{y_2}}](x)$, $x\in\Gamma_1$, i.e.
\ben
I_T(x)=\frac{\mathbf{i}}{4}\int_{H+T}^\infty H_0^{(1)}\left(k_1\sqrt{\left(x_1+\frac{\Lambda}{2}\right)^2+(x_2-\widetilde{y_2})^2}\right)\e^{\mathbf{i}\beta_n\widetilde{y_2}} (1+\mathbf{i}S) \mathrm{d}y_2.
\enn
In view of the Graf's addition theorem of Hankel function with complex variables \cite{W22}, which states that for $|r_0\e^{\mathbf{i}\theta_0}|<|\widetilde{y}_2|$
\be\quad\label{addition theorem of Hankel function}
H_0^{(1)}\left(k_1\sqrt{\left(x_1+\frac{\Lambda}{2}\right)^2+(x_2-\widetilde{y_2})^2}\right)=\sum_{m=-\infty}^\infty H_m^{(1)}(k_1 \widetilde{y_2})J_m(k_1r_0)\e^{\mathbf{i}m\theta_0},
\en
where $r_0=\sqrt{(x_1+\frac{\Lambda}{2})^2+x_2^2}$ and $\theta_0=\arctan\frac{\frac{ST}{P+1}+S(y_2-H-T)}{y_2}$, and the relation $H_{-m}^{(1)}(z)=(-1)^mH_m^{(1)}(z)$, it suffices to estimate the convergence of the integrals
\ben
I_T^{n,m}:=\int_{H+T}^\infty \left|H_m^{(1)}(k_1\widetilde{y_2})\e^{\mathbf{i}\beta_n\widetilde{y_2}}\right|\mathrm{d}y_2,\quad m\geq 0,
\enn
as $T\to\infty$. We first consider the case $n\in\mathcal{P}$ for which $\beta_n>0$. Noting that $\widetilde{y_2}=y_2+\mathbf{i}\int_0^{y_2}\sigma(t)\mathrm{d}t=y_2+\mathbf{i}\frac{ST}{P+1}+\mathbf{i}S(y_2-H-T)$, we perform the change of variable $\xi=y_2-H$, which yields
\ben
I_T^{n,m}=\e^{-\beta_n\frac{ST}{P+1}}\int_T^\infty \left|H_m^{(1)}(z(\xi))\right| \e^{-\beta_nS(\xi-T)}\mathrm{d}\xi,
\enn
 where $z(\xi)=k_1\left(\xi+H+\mathbf{i}b(\xi)\right)$ with $b(\xi)=\frac{ST}{P+1}+S\xi-ST$. Consequently, we have $|z(\xi)|\ge C_{1}k_1\xi$ where $C_{1}=\sqrt{\frac{2S}{P+1}}$. Use of \cite[Eq. (B.7)]{SFFP23} and \cite[Lemma 1]{DY10} gives the estimate
\ben
|H_m^{(1)}(z(\xi))|\leq \frac{\e^{-k_1b(\xi)}}{\sqrt{8k_1 C_{1}}}\frac{2^mP_0(m)}{\Gamma(m-\frac{1}{2})}\xi^{-\frac{1}{2}},\quad \xi\geq T,
\enn
where $P_0(m)$ is a positive-coefficients polynomial of degree $0$. Consequently, 
\ben
|I_T^{n,m}|\leq \e^{-(k_1+\beta_n)\frac{ST}{P+1}}\frac{1}{\sqrt{8k_1C_{1}}}\frac{2^mP_0(m)}{\Gamma(m-\frac{1}{2})}\int_T^\infty \xi^{-\frac{1}{2}}\e^{-(k_1+\beta_n)S(\xi-T)}\mathrm{d}\xi.
\enn
Since $\int_T^\infty \xi^{-\frac{1}{2}}\e^{-(k_1+\beta_n)S(\xi-T)}\mathrm{d}\xi=\frac{\e^{(k_1+\beta_n) ST}}{\sqrt{(k_1+\beta_n)S}}\Gamma(\frac{1}{2},(k_1+\beta_n)ST)$, where $\Gamma(s,x)=\int_x^\infty t^{s-1}\e^{-t}\mathrm{d}t$ is the incomplete Gamma function, we have
\be\qquad\label{bound of I_T^{n,m}}
|I_T^{n,m}|\leq \e^{-(k_1+\beta_n)\frac{ST}{P+1}}\frac{1}{\sqrt{8k_1C_{1}}}\frac{2^mP_0(m)}{\Gamma(m-\frac{1}{2})}\frac{\e^{(k_1+\beta_n)ST}}{\sqrt{(k_1+\beta_n)S}}\Gamma(\frac{1}{2},(k_1+\beta_n)ST).
\en
Using the asymptotic approximation and expansion of the incomplete Gamma function of large $(k_1+\beta_n)ST$ \cite{OLBC10}, we have
\ben
\Gamma(\frac{1}{2},(k_1+\beta_n)ST)=\frac{\e^{-(k_1+\beta_n)ST}}{\sqrt{(k_1+\beta_n)ST}}\left(\sum_{k=0}^{s-1}\frac{u_k}{((k_1+\beta_n)ST)^k}+R_s(\frac{1}{2},(k_1+\beta_n)ST)\right),
\enn
where $u_k=(\frac{1}{2}-1)(\frac{1}{2}-2)\cdots(\frac{1}{2}-k)$ and $R_s(\frac{1}{2},(k_1+\beta_n)ST)=\mathcal{O}(((k_1+\beta_n)ST)^{-s})$. Combining the suitable upper bounds \eqref{bound of I_T^{n,m}} for $I_T^{n,m}$ with the addition theorem \eqref{addition theorem of Hankel function}, we obtain the following estimate
\be\label{bound of I_T(x)1}
\begin{aligned}
|I_T(x)|
&\leq C\e^{-(k_1+\beta_n)\frac{ST}{P+1}}\sum_{m=0}^\infty h_m(r_0)
\end{aligned}
\en
where $h_m(r_0)=|J_m(k_1r_0)|\frac{2^mP_0(m)}{|\Gamma(m-\frac{1}{2})|}$ and the constant $C>0$ is bounded as
\begin{align*}
C&=\sqrt{\frac{1+S^2}{32k_1TS^2C_{1}(k_1+\beta_n)^2}}\left|\sum_{k=0}^{s-1}\frac{u_k}{((k_1+\beta_n)ST)^k}\right|\\
&\le \sqrt{\frac{1+S^2}{32k_1TS^2C_{1}(k_1+\beta_n)^2(1-\frac{1}{(k_1+\beta_n)ST})}},
\end{align*}
in which we have used the estimates that, given $|k_1+\beta_n|ST>1$,
\ben
\begin{aligned}
&\left|\sum_{k=0}^{s-1}\frac{u_k}{((k_1+\beta_n)ST)^k}\right|\leq\sum_{k=0}^{s-1}\frac{|u_k|}{(|k_1+\beta_n|ST)^k}=\sum_{k=0}^{s-1}\frac{(2k)!}{4^kk!(|k_1+\beta_n|ST)^k}\\
&=\sum_{k=0}^{s-1}\begin{pmatrix}
      2k\\k
\end{pmatrix}\left(\frac{1}{4|k_1+\beta_n|ST}\right)^k \leq\sum_{k=0}^\infty\begin{pmatrix}
   2k\\k
\end{pmatrix}\left(\frac{1}{4|k_1+\beta_n|ST}\right)^k=\frac{1}{\sqrt{1-\frac{1}{|k_1+\beta_n|ST}}}.
\end{aligned}
\enn
To establish the exponential decay of the function $I_T$ as $T\to\infty$, it suffices to demonstrate the absolute convergence of the series $\sum_{m=0}^\infty h_m(r_0)$ in \eqref{bound of I_T(x)1}. We achieve this via the ratio test, leveraging the asymptotic behavior of the Bessel function $J_m(k_1r_0)$ for real number $k_1r_0$ and large $m$ in \cite{OLBC10}
\ben
|J_m(k_1r_0)|\sim\frac{1}{\sqrt{2\pi m}}\left(\frac{ek_1r_0}{2m}\right)^m,
\enn
and Stirling's formula 
\ben
\Gamma(m-\frac{1}{2})\sim\frac{\sqrt{2\pi}}{m-\frac{1}{2}}\left(\frac{m-\frac{1}{2}}{\e}\right)^{m-\frac{1}{2}}
\enn
for large $m$. Substituting these asymptotic expressions into $h_m(r_0)$, it follows that
\ben
h_m(r_0)\sim\frac{P_0(m)}{2\pi}\frac{m-\frac{1}{2}}{m}\left(\frac{\e^2k_1r_0}{m(m-\frac{1}{2})}\right)^m\sqrt{\frac{m-\frac{1}{2}}{\e}}.
\enn
Therefore, evaluating the limit of consecutive term ratios implies that
\ben
\lim_{m\to\infty}\frac{h_{m+1}(r_0)}{h_m(r_0)}=\lim_{m\to\infty}\left\{\frac{m(m+\frac{1}{2})^{\frac{1}{2}}\e^2k_1r_0}{(m+1)^2(m-\frac{1}{2})^\frac{3}{2}}\left(\frac{m(m-\frac{1}{2})}{(m+1)(m+\frac{1}{2})}\right)^m\right\}=0,
\enn
which completes the proof for $n\in\mathcal{P}$.

Next, we consider the case of $n\in\mathcal{Q}$, that is $\beta_n\in\mathbf{i}\mathbb{R}^+$. In this case, the integral $I_T^{n,m}$ becomes
\ben
\begin{aligned}
I_T^{n,m}&=\int_T^\infty \left|H_m^{(1)}(z(\xi))\right|\cdot\left|\e^{-|\beta_n|(\xi+H)-\mathbf{i}|\beta_n|S(\xi-T)-\mathbf{i}|\beta_n|\frac{ST}{P+1}}\right|\mathrm{d}\xi\\
&=\int_T^\infty \left|H_m^{(1)}(z(\xi))\right|\e^{-|\beta_n|(\xi+H)}\mathrm{d}\xi,
\end{aligned}
\enn
and we utilize the property of the Hankel function $H_m^{(1)}(z(\xi))$ to obtain
\be\label{bound of I_T^{n,m} for n in Q}
\begin{aligned}
I_T^{n,m}&\leq \frac{\e^{-k_1\frac{ST}{P+1}}}{4}\frac{1}{\sqrt{8k_1C_{1}}}\frac{2^mP_0(m)}{\Gamma(m-\frac{1}{2})}\int_T^\infty \xi^{-\frac{1}{2}}\e^{-|\beta_n|(\xi+H)-k_1S(\xi-T)}\mathrm{d}\xi\\
&=\frac{\e^{-k_1\frac{ST}{P+1}}}{4}\frac{1}{\sqrt{8k_1C_{1}}}\frac{2^mP_0(m)}{\Gamma(m-\frac{1}{2})}\frac{\e^{-|\beta_n|H+k_1ST}}{\sqrt{|\beta_n|+k_1S}}\Gamma(\frac{1}{2},(|\beta_n|+k_1S)T).
\end{aligned}
\en
Combining the suitable upper bounds \eqref{bound of I_T^{n,m} for n in Q} for $I_T^{n,m}$ with the addition theorem \eqref{addition theorem of Hankel function}, we obtain the following estimate
\ben
\begin{aligned}
|I_T(x)|&=\frac{\e^{-k_1\frac{ST}{P+1}}}{4}\sqrt{\frac{1+S^2}{8k_1C_1}}\frac{\e^{-|\beta_n|H+k_1ST}}{((|\beta_n|+k_1S))^{\frac{1}{2}}}\Gamma(\frac{1}{2},(|\beta_n|+k_1S)T)\left|\sum_{m=-\infty}^\infty J_m(k_1r_0)\frac{2^mP_0(m)}{\Gamma(m-\frac{1}{2})}\right|\\
&\leq\frac{\e^{-k_1\frac{ST}{P+1}}}{2}\sqrt{\frac{1+S^2}{8k_1C_1}}\frac{\e^{-|\beta_n|(H+T)}}{(|\beta_n|+k_1S)^{\frac{1}{2}}}\Gamma(\frac{1}{2},(|\beta_n|+k_1S)T)\sum_{m=0}^\infty h_m(r_0),\\
&\leq \e^{-k_1\frac{ST}{P+1}-|\beta_n|(H+T)}\sqrt{\frac{1+S^2}{32k_1C_1}}\frac{T^{-\frac{1}{2}}\sum_{k=0}^{s-1}\frac{u_k}{((|\beta_n|+k_1S)T)^k}}{|\beta_n|+k_1S}\sum_{m=0}^\infty h_m(r_0),
\end{aligned}
\enn
and hence the desired result follows for $n\in\mathcal{Q}$.
\end{proof}
\begin{remark}
\label{convergence2}
As $T\to\infty$, the convergence of the PML-complement integral operator $\widetilde{D}_1^{\Gamma_1,\Gamma_2^+}[\e^{\pm\mathbf{i}\beta_n\widetilde{y}_2}]$, $\widetilde{K}_1^{\Gamma_1,\Gamma_2^+}[\e^{\pm\mathbf{i}\beta_n\widetilde{y}_2}]$ and $\widetilde{N}_1^{\Gamma_1,\Gamma_2^+}[\e^{\pm\mathbf{i}\beta_n\widetilde{y}_2}]$ for $n\in\mathcal{P}\cup\mathcal{Q}$ can be established analogously to that of $\widetilde{S}_1^{\Gamma_1,\Gamma_2^+}[\e^{\pm \mathbf{i}\beta_n\widetilde{y}_2}]$. This generalization relies on applying the following Graf's addition theorem for Hankel function of order $\nu$ with complex variables \cite{W22}, which holds true under the condition $|r_0\e^{\mathbf{i}\theta_0}|<|\widetilde{y}_2|$
\ben
H_\nu\left(k_1\sqrt{\left(x_1+\frac{\Lambda}{2}\right)^2+(x_2-\widetilde{y_2})^2}\right)\e^{i\nu\theta_1}=\sum_{m=-\infty}^\infty H_{\nu+m}(k_1\widetilde{y}_2)J_m(k_1r_0)\e^{\mathbf{i}m\theta_0}
\enn
where $\theta_1=\arccos\frac{\real\left(x_2\sqrt{y_2^2\left(\frac{ST}{P+1}+Sy_2-SH-ST\right)^2}\e^{-\mathbf{i}\theta_0}\right)}{\sqrt{\left(x_1+\frac{\Lambda}{2}\right)^2+x_2^2}\sqrt{y_2^2+\left(\frac{ST}{P+1}+Sy_2-SH-ST\right)^2}}$.
\end{remark}

\begin{remark}
The convergence analysis in Theorem \ref{convergence1} reveals a subtle issue underlying the remarkable failure of the preliminary truncated PML-BIE \eqref{truncated PML-BIE} at certain frequencies. According to Theorem \ref{convergence1} and Remark \ref{convergence2}, not only the tail integrals $\widetilde{\mathbb{T}}^\pm[\e^{\pm\mathbf{i}\beta_n\widetilde{y}_2}]$ for $\beta_n\in\mathbb{R}^+$ ($n\in\mathcal{P}$) decay exponentially fast as $T$ increases, but also $\widetilde{\mathbb{T}}^\pm[\e^{\mp\mathbf{i}\beta_n}\widetilde{y}_2]$ in \eqref{tail integral terms of C} as long as $\beta_n\in\mathbb{R}\cap\{\beta_n\neq k_1\}$. Indeed, the latter tends to zero as $\mathcal{O}(\e^{-(k_1-\beta_n)\frac{ST}{P+1}})$. Similarly, the tail integrals $\widetilde{\mathbb{T}}^\pm[\e^{\pm\mathbf{i}\beta_n\widetilde{y}_2}]$ for $\beta_n\in\mathbf{i}\mathbb{R}^+$ ($n\in\mathcal{Q}$) also decay exponentially fast as $T$ increases, and so do $\widetilde{\mathbb{T}}^\pm[\e^{\mp\mathbf{i}\beta_n\widetilde{y}_2}]$ in \eqref{tail integral terms of C}. In fact, these terms tend to zero as $\mathcal{O}(\e^{-k_1\frac{ST}{P+1}-|\beta_n|(H+T)})$ and $\mathcal{O}(\e^{-k_1\frac{ST}{P+1}+|\beta_n|(H+T)})$, respectively. For a fixed $T>0$, this fact renders $C_n^\pm\widetilde{\mathbb{T}}^\pm[\e^{\mp\mathbf{i}\beta_n\widetilde{y}_2}]$ negligible for $\beta_n\in\mathbb{R}^+\setminus\{\beta_n=k_1\}$ and $\beta_n\in\mathbf{i}\mathbb{R}^+$ in \eqref{tail integral terms of C}, irrespective of the actual value of the coefficients $C_n^\pm$, thus making the condition $\psi_{p,C}=0$ used in \eqref{truncated PML-BIE} insufficient to rigorously enforce the desired radiation condition $C_n^\pm=0$. In other words, the equation $\psi_{p,C}=0$ for the vanishing coefficients $C_n^\pm$ become ill-conditioned in practice, allowing the non-propagation modes to persist and pollute the approximate solution of \eqref{truncated PML-BIE}. 
\end{remark}

\section{Modified PML-BIE}
\label{sec:4}
This section proposes a modified PML-BIE method to ensure the uniformly high accuracy for all frequencies. The newly derived method is based on the idea of combining the Rayleigh's expansion with the PML-transformed integral representation of upgoing and downgoing modes solutions whose efficiency has been shown in \cite{SFFP23} for the windowed Green function method. We first consider the case $\mathcal{N}=\emptyset$ in Sections~\ref{sec:4.1}--\ref{sec:4.3} and the corresponding method for $\mathcal{N}\ne \emptyset$ will be discussed in Section~\ref{sec:4.4}.

\subsection{Approximation of the complement operators and tail integrals}
\label{sec:4.1}

As shown in Section~\ref{sec:3}, the tail integrals $\psi_{p,C}$ ($p=1,\cdots,4$) cannot be neglected directly and hence, are required to be approximated in an appropriate manner. Using the definition of $\mathcal{C}_\delta$ in (\ref{setC}), we can first truncate the infinite series in the tail integrals $\psi_{p,C}$ using a finite number of terms $n\in\mathcal{C}_\delta$ as
\ben
\psi_{p,C}\approx\sum_{n\in\mathcal{C}_\delta}\left\{C_n^+\Psi_{n,p}^++C_n^-\Psi_{n,p}^-\right\}\quad \mbox{on}\quad (I_{2\pi})^2 \times (I_{H+T})^2,
\enn
where the functions $\Psi_{n,p}^\pm$ are given by
\be\label{semi-infinite integrals for non-propagation modes}
\Psi_{n,p}^\pm=\mathbb{\widetilde{T}}_{p,3}^\pm[\widetilde{\phi}_{n,3}^\mp]+\mathbb{\widetilde{T}}_{p,4}^\pm[\widetilde{\phi}_{n,4}^\mp] \quad \mbox{on}\quad (I_{2\pi})^2 \times (I_{H+T})^2,
\en
in which the semi-infinite integral operators $\mathbb{\widetilde{T}}_{p,q}^\pm[\widetilde{\phi}_{n,q}^\mp],q=3,4,$ either cannot be calculated in closed form or simply diverge due to the evaluation of improper integrals over the unbounded curves $(H+T,\infty)$ or $(-\infty,-H-T)$ (i.e., $\Gamma_2^\pm$ and $\Gamma_3^\pm$). To overcome this difficulty, one can use the following complementary integrals
\be\label{complementary integrals over semi-infinite boundaries}
\mathbb{\widetilde{T}}_{p,3}^\mp[\widetilde{\phi}_{n,3}^\mp]+\mathbb{\widetilde{T}}_{p,4}^\mp[\widetilde{\phi}_{n,4}^\mp].
\en
As shown in Theorem~\ref{convergence1} and Remark~\ref{convergence2}, the terms $\mathbb{\widetilde{T}}_{p,q}^\mp[\widetilde{\phi}_{n,q}^\mp]$ in \eqref{complementary integrals over semi-infinite boundaries} decay exponentially for all $n\in \mathcal{C}_\delta$. Hence, adding \eqref{complementary integrals over semi-infinite boundaries} into \eqref{semi-infinite integrals for non-propagation modes} yields
\ben
\Psi_{n,p}^\pm\approx\mathbb{\widetilde{T}}_{p,3}^c[\widetilde{\phi}_{n,3}^\mp]+\mathbb{\widetilde{T}}_{p,4}^c[\widetilde{\phi}_{n,4}^\mp] \quad \mbox{on}\quad (I_{2\pi})^2 \times (I_{H+T})^2,
\enn
where $\mathbb{\widetilde{T}}_{p,q}^c[\widetilde{\phi}_{n,q}^\mp]=\mathbb{\widetilde{T}}_{p,q}^+[\widetilde{\phi}_{n,q}^\mp]+\mathbb{\widetilde{T}}_{p,q}^-[\widetilde{\phi}_{n,q}^\mp]=\mathbb{\widetilde{T}}_{p,q}[\widetilde{\phi}_{n,q}^\mp]-\mathbb{\widetilde{T}}^b_{p,q}[\widetilde{\phi}_{n,q}^\mp]$, for $p=1,\ldots,4$ and $q=3,4$. Therefore, the tail integrals $\psi_{p,C}$ can be further approximated as
\be\label{computational approximation of tail integrals}
\Psi_{n,p}^\pm\approx -\widehat{\Psi}_{n,p}^\pm- \mathbb{\widetilde{T}}_{p,3}^b[\widetilde{\phi}_{n,3}^\mp]-\mathbb{\widetilde{T}}_{p,4}^b[\widetilde{\phi}_{n,4}^\mp]\quad \mbox{on}\quad (I_{2\pi})^2 \times (I_{H+T})^2, 
\en
where
\be\label{infinite integrals over Gamma_2}
\widehat{\Psi}_{n,p}^\pm:=-\widetilde{\mathbb{T}}_{p,3}[\widetilde{\phi}_{n,3}^\mp]-\widetilde{\mathbb{T}}_{p,4}[\widetilde{\phi}_{n,4}^\mp],
\en
are the infinite integrals over the infinite domain $\mathbb{R}$ (i.e., the infinite boundaries $\Gamma_2$ and $\Gamma_3$) whose simple closed form can be derived as follows.

Note that the Rayleigh modes $u_n^\pm$ for $n\in\mathcal{P}$ and $u_n$ for $n\in\mathcal{N}$ admit the following representation
\ben
{u}_n^\pm={\mathscr{D}}_1^{\Gamma_2}[{u}_n^\pm]-{\mathscr{S}}_1^{\Gamma_2}[{\pa}_{\nu}{u}_n^\pm]-{\mathscr{D}}_1^{\Gamma_3}[{u}_n^\pm]+{\mathscr{S}}_1^{\Gamma_3}[{\pa}_{\nu}{u}_n^\pm]\quad\mbox{in}\quad U,
\enn
which implies that the PML-transformed Rayleigh modes $\widetilde{u}_n^\pm$ satisfy
\be\label{Solrep for Rayleigh modes}
\widetilde{u}_n^\pm=\widetilde{\mathscr{D}}_1^{\Gamma_2}[\widetilde{u}_n^\pm]-\widetilde{\mathscr{S}}_1^{\Gamma_2}[\widetilde{\pa}_{\nu}\widetilde{u}_n^\pm]-\widetilde{\mathscr{D}}_1^{\Gamma_3}[\widetilde{u}_n^\pm]+\widetilde{\mathscr{S}}_1^{\Gamma_3}[\widetilde{\pa}_{\nu}\widetilde{u}_n^\pm]\quad\mbox{in}\quad U.
\en
Then it can be derived that
\be\label{closed form}
\begin{aligned}
\widehat{\Psi}_{n,p}^\pm&=-\begin{cases}
   (\zeta\widetilde{\mathcal{K}}_{1,p}^{\Gamma_1,\Gamma_3}-\widetilde{\mathcal{K}}_{1,p}^{\Gamma_1,\Gamma_2})[\widetilde{\phi}_{n,3}^\mp](x)+(\widetilde{\mathcal{J}}_{1,p}^{\Gamma_1,\Gamma_2}-\zeta\widetilde{\mathcal{J}}_{1,p}^{\Gamma_1,\Gamma_3})[\widetilde{\phi}_{n,4}^\mp](x),&p=1,2,\\
   (\zeta^2\widetilde{\mathcal{K}}_{1,p}^{\Gamma_2,\Gamma_3}-\widetilde{\mathcal{K}}_{1,p}^{\Gamma_3,\Gamma_2})[\widetilde{\phi}_{n,3}^\mp](x)+(\widetilde{\mathcal{J}}_{1,p}^{\Gamma_3,\Gamma_2}-\zeta^2\widetilde{\mathcal{J}}_{1,p}^{\Gamma_2,\Gamma_3})[\widetilde{\phi}_{n,4}^\mp](x),&p=3,4,
   \end{cases}\\
&=\begin{cases}
   \phi_{n,p}^\mp,&p=1,2,\\
   \zeta\widetilde{\phi}_{n,p}^\mp,&p=3,4.
\end{cases}
\end{aligned}
\en
where
\ben
\widetilde{\mathcal{J}}_{j,p}^{\Gamma_l,\Gamma_i}:=\begin{cases}
   \widetilde{S}_j^{\Gamma_l,\Gamma_i},&p=1,3,\\
   \widetilde{K}_j^{\Gamma_l,\Gamma_i},&p=2,4,
\end{cases}\quad\mathrm{and}\quad\widetilde{\mathcal{K}}_{j,p}^{\Gamma_l,\Gamma_i}:=\begin{cases}
   \widetilde{D}_j^{\Gamma_l,\Gamma_i},&p=1,3,\\
   \widetilde{N}_j^{\Gamma_l,\Gamma_i},&p=2,4.
\end{cases}
\enn
For $\beta_n\in\mathbb{R}^+\backslash\{\beta_n=k_1\}$, we have derived a computable approximation \eqref{computational approximation of tail integrals} of the mode integrals $\Psi_{n,p}^\pm$, whose error that decay exponentially as $T\to\infty$. Such an approximation consists of the closed-form expression \eqref{closed form} and the finite-domain PML integrals in \eqref{computational approximation of tail integrals} that can be evaluated numerically. Since the right-hand side of $\eqref{computational approximation of tail integrals}$ constitutes an analytic function of $\beta_n$, the corresponding computable expressions for $\beta_n \in \mathbf{i}\mathbb{R}^+ \cup \{\beta_n=k_1\}$ are systematically obtained by the principle of analytic extension. This process allows us to rigorously extend the formula derived for the propagating modes to these $\beta_n$ values. Finally, combining the following notation 
\be\label{Psi_tilde}
\widetilde{\bs\Psi}_n^\pm:=-\begin{bmatrix}
   \phi_{n,1}^\mp\\
   \phi_{n,2}^\mp\\
   \zeta\widetilde{\phi}_{n,3}^\mp\\
   \zeta\widetilde{\phi}_{n,4}^\mp
\end{bmatrix}-\mathbb{\widetilde{T}}^b\widehat{\bs\Psi}_n^\pm \quad\mbox{with}\quad \widehat{\bs\Psi}_n^\pm=\begin{bmatrix}
   0\\ 0\\ \widetilde{\phi}_{n,3}^\mp\\ \widetilde{\phi}_{n,4}^\mp
\end{bmatrix}
\en
with \eqref{computational approximation of tail integrals}, we can rewrite the PML-BIE system (\ref{equivalent PML-BIE}) as the following approximated form
\be\label{corrected PML-BIE base on finite rank operator}
\begin{aligned}
(\mathbb{E}+\mathbb{\widetilde{T}}^b)(\widetilde{\bs \phi})+\sum_{n\in\mathcal{C}_\delta}\left\{C_n^+\widetilde{\bs\Psi}_n^++C_n^-\widetilde{\bs\Psi}_n^-\right\}=\bs\phi^\inc\quad \mbox{on}\quad (I_{2\pi})^2 \times (I_{H+T})^2,
\end{aligned}
\en
associated with unknowns $\widetilde{\bs \phi}$ and $C_n^\pm$ in the case $\mathcal{N}=\emptyset$. The additional equations required to relate the coefficients $C_n^\pm$, $n\in\mathcal{C}_\delta$, to the vector densities $\widetilde{\bs\phi}$ can be obtained from \eqref{weak RayleighExp} and \eqref{weaker radiation condition}, which yields
\be\label{coefficients with non-propagation modes in finite modes}
C_n^\pm=\frac{\mp 1}{2\mathbf{i}\beta_n\Lambda}\e^{-\textbf{i}\beta_n h}\int_{-\frac{\Lambda}{2}}^{\frac{\Lambda}{2}}\left( \pa_{x_2}u^{\sct}(x_1,\pm h)\mp \textbf{i}\beta_n u^{\sct}(x_1,\pm h) \right)\e^{-\textbf{i}\alpha_n x_1}\mathrm{d}x_1,
\en
for $n\in\mathcal{C}_\delta$ and any $\max\{h^+,-h^-\}<h\le H$. It can be viewed that the evaluation of (\ref{coefficients with non-propagation modes in finite modes}) still relies on appropriate approximation of the scattered field on $(-\Lambda/2,\Lambda/2)\times \{\pm h\}$ (or more generally, in $\Omega_1^b$) which will be discussed in the next subsection.

\subsection{Modified approximation of the scattered field and modified PML-BIE}
\label{sec:4.2}

Note that the Dirichlet and Neumann traces of $u^\sct$ on $\Gamma_2$ can be approximated by $\widetilde{\phi}_q$ on $\Gamma_2^b$ and $\sum_{n\in\mathcal{C}_\delta}\{ C_n^\pm\widetilde{\phi}_{n,q}^\mp \}$ on $\Gamma_2^\pm$ for $q=3,4$. Through smoothly merging the scattered field in the PML domain with the corresponding Rayleigh waves in the PML complementary domain, we arrive at the following modified approximation of the scattered field
\be\label{correction of u^sct using modal term}
\begin{aligned}
u^{\sct}(x)&\approx\widetilde{\mathscr{D}}_1^{\Gamma_1^b}[\widetilde{\phi}_1](x)-\eta\widetilde{\mathscr{S}}_1^{\Gamma_1^b}[\widetilde{\phi}_2](x)\\
&\quad+(\widetilde{\mathscr{D}}_1^{\Gamma_2^b}-\zeta\widetilde{\mathscr{D}}_1^{\Gamma_3^b})[\widetilde{\phi}_3](x)-(\widetilde{\mathscr{S}}_1^{\Gamma_2^b}-\zeta\widetilde{\mathscr{S}}_1^{\Gamma_3^b})[\widetilde{\phi}_4](x) \\
&\quad+\sum_{n\in\mathcal{C}_\delta}C_n^+\big\{ (\widetilde{\mathscr{D}}_1^{\Gamma_2^+}-\zeta\widetilde{\mathscr{D}}_1^{\Gamma_3^+})[\widetilde{\phi}_{n,3}^-](x)+(\widetilde{\mathscr{S}}_1^{\Gamma_2^+}-\zeta\widetilde{\mathscr{S}}_1^{\Gamma_3^+})[\widetilde{\phi}_{n,4}^-](x) \big\},\\
&\quad+\sum_{n\in\mathcal{C}_\delta}C_n^-\big\{ (\widetilde{\mathscr{D}}_1^{\Gamma_2^-}-\zeta\widetilde{\mathscr{D}}_1^{\Gamma_3^-})[\widetilde{\phi}_{n,3}^+](x)+(\widetilde{\mathscr{S}}_1^{\Gamma_2^-}-\zeta\widetilde{\mathscr{S}}_1^{\Gamma_3^-})[\widetilde{\phi}_{n,4}^+](x) \big\},&\quad x\in\Omega_1, 
\end{aligned}
\en
where the semi-infinite layer potentials $\widetilde{\mathscr{S}}_1^{\Gamma_2^\pm}$, 
$\widetilde{\mathscr{D}}_1^{\Gamma_2^\pm}$ and $\widetilde{\mathscr{S}}_1^{\Gamma_3^\pm}$, $\widetilde{\mathscr{D}}_1^{\Gamma_3^\pm}$ are defined on $\Gamma_2^\pm$ and $\Gamma_3^\pm$, respectively.

To produce a computable approximation of the modal terms on semi-infinite boundaries in (\ref{correction of u^sct using modal term}), we follow again the idea of using complementary integrals discussed in Section~\ref{sec:4.1} to note that, for a target point $x\in\Omega_1^b$, the integrals
\ben
(\widetilde{\mathscr{D}}_1^{\Gamma_2^\pm}-\zeta\widetilde{\mathscr{D}}_1^{\Gamma_3^\pm})[\widetilde{\phi}_{n,3}^\mp](x)+(\widetilde{\mathscr{S}}_1^{\Gamma_2^\pm}-\zeta\widetilde{\mathscr{S}}_1^{\Gamma_3^\pm})[\widetilde{\phi}_{n,4}^\mp](x)
\enn
can be effectively approximated by
\ben
(\widetilde{\mathscr{D}}_1^{\Gamma_2^c}-\zeta\widetilde{\mathscr{D}}_1^{\Gamma_3^c})[\widetilde{\phi}_{n,3}^\mp](x)-(\widetilde{\mathscr{S}}_1^{\Gamma_2^c}-\zeta\widetilde{\mathscr{S}}_1^{\Gamma_3^c})[\widetilde{\phi}_{n,4}^\mp](x)
\enn
with error
\ben
(\widetilde{\mathscr{D}}_1^{\Gamma_2^\mp}-\zeta\widetilde{\mathscr{D}}_1^{\Gamma_3^\mp})[\widetilde{\phi}_{n,3}^\mp](x)+(\widetilde{\mathscr{S}}_1^{\Gamma_2^\mp}-\zeta\widetilde{\mathscr{S}}_1^{\Gamma_3^\mp})[\widetilde{\phi}_{n,4}^\mp](x)
\enn
that exponentially converges to zero for $n\in\mathcal{P}$ or $n\in\mathcal{Q}$ as $T\to\infty$.

Since $\widetilde{\phi}_{n,p}^\pm=\phi_{n,p}^\pm$ for $\max\{h^+,-h^-\}<\pm x_2\leq H$, we are in the position to derive the finite rank operators in physical domain. Denoting the modified densities $\widetilde{\phi}_{q}^{\modi},q=1,\ldots,4,$ to form the vector density
\be\label{corrected vector density}
\widetilde{\bs \phi}^{\modi}=\widetilde{\bs\phi}-\sum_{n\in\mathcal{C}_\delta}\begin{Bmatrix}
   C_n^+\widehat{\bs\Psi}_n^++C_n^-\widehat{\bs\Psi}_n^-\\
\end{Bmatrix},
\en
we define $\widetilde{u}^\sct_{\modi}$ as
\be\label{corrected PML approximation of u^sct using Rayleigh modes}
\begin{aligned}
\widetilde{u}^{\sct}_{\modi}(x)&=\widetilde{\mathscr{D}}_1^{\Gamma_1^b}[\widetilde{\phi}_1^\modi](x)-\eta\widetilde{\mathscr{S}}_1^{\Gamma_1^b}[\widetilde{\phi}_2^\modi](x)+(\widetilde{\mathscr{D}}_1^{\Gamma_2^b}-\zeta\widetilde{\mathscr{D}}_1^{\Gamma_3^b})[\widetilde{\phi}_3^\modi](x)\\
&\quad-(\widetilde{\mathscr{S}}_1^{\Gamma_2^b}-\zeta\mathscr{S}_1^{\Gamma_3^b})[\widetilde{\phi}_4^\modi](x) +\sum_{n\in\mathcal{C}_\delta}\left\{ C_n^+\widetilde{u}_n^-+C_n^-\widetilde{u}_n^+ \right\},\quad x\in\Omega_1^b, 
\end{aligned}
\en
where the last two terms are obtained by \eqref{Solrep for Rayleigh modes}. Substituting \eqref{corrected PML approximation of u^sct using Rayleigh modes} into \eqref{coefficients with non-propagation modes in finite modes}, we can get an expression of $C_n^\pm,n\in\mathcal{C}_\delta$ only in terms of $\widetilde{\bs\phi}^{\modi}$ as
\be\label{coefficients with non-propagation modes in finite modes modified}
C_n^\pm=\pm\frac{\e^{\mathbf{i}\beta_n h}}{2\mathbf{i}\beta_n}\mathbb{L}_n^\pm[\widetilde{\bs\phi}^{\modi}],\quad n\in\mathcal{C}_\delta,
\en
where we define the functionals:
\be\label{PML-BIE functionals}
\begin{aligned}
\mathbb{L}_n^\pm[\widetilde{\bs\phi}^{\modi}](x):&=\frac{1}{\Lambda}\int_{-\frac{\Lambda}{2}}^{\frac{\Lambda}{2}}\bigg\{ (\widetilde{\pa}_{x_2}\widetilde{\mathscr{D}}_1^{\Gamma_1^b}\mp\mathbf{i}\beta_n\widetilde{\mathscr{D}}_1^{\Gamma_1^b})[\widetilde{\phi}_1^{\modi}]-\eta(\widetilde{\pa}_{x_2}\widetilde{\mathscr{S}}_1^{\Gamma_1^b}\mp\mathbf{i}\beta_n\widetilde{\mathscr{S}}_1^{\Gamma_1^b})[\widetilde{\phi}_2^{\modi}]\\
&\quad+\left(\widetilde{\pa}_{x_2}\widetilde{\mathscr{D}}_1^{\Gamma_2^b}\mp\mathbf{i}\beta_n\widetilde{\mathscr{D}}_1^{\Gamma_2^b}-\zeta(\widetilde{\pa}_{x_2}\widetilde{\mathscr{D}}_1^{\Gamma_3^b}\mp\mathbf{i}\beta_n\widetilde{\mathscr{D}}_1^{\Gamma_3^b})\right)[\widetilde{\phi}_3^{\modi}]\\
&\quad-\left(\widetilde{\pa}_{x_2}\widetilde{\mathscr{S}}_1^{\Gamma_2^b}\mp\mathbf{i}\beta_n\widetilde{\mathscr{S}}_1^{\Gamma_2^b}-\zeta(\widetilde{\pa}_{x_2}\widetilde{\mathscr{S}}_1^{\Gamma_3^b}\mp\beta_n\widetilde{\mathscr{S}}_1^{\Gamma_3^b})\right)[\widetilde{\phi}_4^{\modi}] \bigg\}(x)\e^{-\mathbf{i}\alpha_n x_1}\mathrm{d}x_1,
\end{aligned}
\en

Hence, using \eqref{corrected vector density}, \eqref{coefficients with non-propagation modes in finite modes modified}-\eqref{PML-BIE functionals}, the BIEs \eqref{corrected PML-BIE base on finite rank operator} can be rewritten as the following matrix form:
\be\label{Modified PML-BIE except for RW-anomalies}
\begin{aligned}
(\mathbb{E}+\widetilde{\mathbb{T}}^b+\widetilde{\mathbb{M}})(\widetilde{\bs\phi}^{\modi})=\bs\phi^\inc \quad\mbox{on}\quad (I_{2\pi})^2 \times (I_{H+T})^2,
\end{aligned}
\en
for the corrected vector density $\widetilde{\bs\phi}^{\modi}$ defined in (\ref{corrected vector density}), and where letting $\widetilde{\bs\Phi}_n^\pm=\left[\widetilde{\phi}_{n,1}^\mp, \widetilde{\phi}_{n,2}^\mp, 0, 0 \right]^\top$, the modified finite-rank operator $\widetilde{\mathbb{M}}$ can be defined as
\ben
\widetilde{\mathbb{M}}:=\sum_{n\in\mathcal{C}_\delta}\frac{\e^{\mathbf{i}\beta_nh}}{2\mathbf{i}\beta_n}\left\{ \widetilde{\bs\Phi}_n^-\mathbb{L}_n^--\widetilde{\bs\Phi}_n^+\mathbb{L}_n^+ \right\}.
\enn
In addition, the modified approximate scattered field $\widetilde{u}_\modi^\sct$ reads:
\be\label{Final modified PML approximation of u^sct except for RW-anomalies}
\begin{aligned}
\widetilde{u}^{\sct}_{\modi}(x)&=\widetilde{\mathscr{D}}_1^{\Gamma_1^b}[\widetilde{\phi}_1^\modi](x)-\eta\widetilde{\mathscr{S}}_1^{\Gamma_1^b}[\widetilde{\phi}_2^\modi](x)\\
&\quad+(\widetilde{\mathscr{D}}_1^{\Gamma_2^b}-\zeta\widetilde{\mathscr{D}}_1^{\Gamma_3^b})[\widetilde{\phi}_3^\modi](x)-(\widetilde{\mathscr{S}}_1^{\Gamma_2^b}-\zeta\mathscr{S}_1^{\Gamma_3^b})[\widetilde{\phi}_4^\modi](x)\\
&\quad+\frac{1}{2\mathbf{i}}\sum_{n\in\mathcal{C}_\delta}\frac{\e^{i\beta_nh}}{\beta_n}\left\{ \widetilde{u}_n^-\mathbb{L}_n^+[\widetilde{\bs\phi}^\modi]-\widetilde{u}_n^+\mathbb{L}_n^-[\widetilde{\bs\phi}^\modi] \right\}(x),\quad x\in\Omega_1^b.
\end{aligned}
\en

\subsection{Quasi-periodicity to avoid ``corner singularity"}
\label{sec:4.3}

It is worth to note that the functionals $\mathbb{L}_n^\pm$ defined in \eqref{PML-BIE functionals} involve integration along the boundaries $\Gamma_i^b,i=2,3$, which are intersected by the horizontal line segments $(-\frac{\Lambda}{2},\frac{\Lambda}{2})\times\{\pm h\}$ which, however, will result into possible corner singularity at $(\pm \frac{\Lambda}{2},\pm h)$. To avoid the corner singularity of the integrals, we can leverage the quasi-periodicity condition satisfied by the scattered field and express it by means of the Green's representation formula applied within a $3\Lambda$-period supercell. To do this, we define
\ben
\Gamma_{(1,m)}^b:=\Gamma_1^b+m\Lambda(1,0)
\enn
for the obstacles' boundaries, and note that $\Gamma_{(1,0)}^b=\Gamma_1^b$ is the same boundary of unit cell. Accordingly, we can define 
\ben
\Gamma_{(2,m)}^b:=\Gamma_2+m\Lambda(1,0),\quad \Gamma_{(3,m)}^b:=\Gamma_3+m\Lambda(1,0)
\enn
for the sub-boundaries of  $\Gamma_2^b$ and $\Gamma_3^b$. Subsequently, letting the index set of $\mathcal{I}_{m}=\{ -1,0,1 \}$, the functionals can be expressed as
\ben
\begin{aligned}
& \mathbb{L}_n^\pm[\widetilde{\bs\phi}^\modi](x)\\
&=\frac{1}{\Lambda}\int_{-\frac{\Lambda}{2}}^{\frac{\Lambda}{2}}\bigg\{ (\widetilde{\pa}_{x_2}\widetilde{\mathscr{D}}_1^{\widetilde{\Gamma}_1^b}\mp\mathbf{i}\beta_n\widetilde{\mathscr{D}}_1^{\widetilde{\Gamma}_1^b})[\widetilde{\phi}_1^\modi]-\eta(\widetilde{\pa}_{x_2}\widetilde{\mathscr{S}}_j^{\widetilde{\Gamma}_1^b}\mp\mathbf{i}\beta_n\widetilde{\mathscr{S}}_1^{\widetilde{\Gamma}_1^b})[\widetilde{\phi}_2^\modi]+\\
&\quad\left(\zeta^{-1}(\widetilde{\pa}_{x_2}\mathscr{D}_1^{\Gamma_{(2,-1)}^b}\mp\mathbf{i}\beta_n\widetilde{\mathscr{D}}_1^{\Gamma_{(2,-1)}^b})-\zeta^2(\widetilde{\pa}_{x_2}\widetilde{\mathscr{D}}_1^{\Gamma_{(3,1)}^b}\mp\mathbf{i}\beta_n\widetilde{\mathscr{D}}_1^{\Gamma_{(3,1)}^b})\right)[\widetilde{\phi}_3^\modi]-\\
&\quad\left(\zeta^{-1}(\widetilde{\pa}_{x_2}\widetilde{\mathscr{S}}_1^{\Gamma_{(2,-1)}^b}\mp\mathbf{i}\beta_n\mathscr{S}_1^{\Gamma_{(2,-1)}^b})-\zeta^2(\widetilde{\pa}_{x_2}\widetilde{\mathscr{S}}_1^{\Gamma_{(3,1)}^b}\mp\mathbf{i}\beta_n\widetilde{\mathscr{S}}_1^{\Gamma_{(3,1)}^b})\right)[\widetilde{\phi}^\modi] \bigg\}(x)\e^{-\mathbf{i}\alpha_n x_1}\mathrm{d}x_1
\end{aligned}
\enn
in term of the layer potential: $\widetilde{\mathscr{D}}_1^{\widetilde{\Gamma}_1^b}$ and $\widetilde{\mathscr{S}}_1^{\widetilde{\Gamma}_1^b}$ associated with $\widetilde{\Gamma}_1^b:=\cup_{m\in\mathcal{I}_m}\Gamma_{(1,m)}^b$; $\widetilde{\mathscr{D}}_1^{\Gamma_{(2,-1)}^b}$ and $\widetilde{\mathscr{S}}_1^{\Gamma_{(2,-1)}^b}$ associated with $\Gamma_{(2,-1)}^b$; $\widetilde{\mathscr{D}}_1^{\Gamma_{(3,1)}^b}$ and $\widetilde{\mathscr{S}}_1^{\Gamma_{(3,1)}^b}$ associated with $\Gamma_{(3,1)}^b$. Moreover, it should be underlined that
\ben
\widetilde{\mathscr{F}}_1^{\widetilde{\Gamma}_1^b}:=\sum_{m\in\mathcal{I}_{m}}\zeta^m\widetilde{\mathscr{F}}_1^{\Gamma_{(1,m)}^b}
\enn
where $\widetilde{\mathscr{F}}$ can be any potential or operator in $\{ \widetilde{\mathscr{S}},\widetilde{\mathscr{D}} \}$.

\subsection{Extension to RW-anomaly}
\label{sec:4.4}
To extend \eqref{Modified PML-BIE except for RW-anomalies} to the challenging RW-anomaly case, we resort to L'H\^{o}pital's rule to evaluate the modified finite rank operator $\widetilde{\mathbb{M}}$ in the limit as $\beta_n\to0$
\ben
\begin{aligned}
\lim_{\beta_n\to0}\widetilde{\mathbb{M}}&=\frac{1}{2\mathbf{i}}\lim_{\beta_n\to0}\sum_{n\in\mathcal{C}_\delta}\frac{\e^{\mathbf{i}\beta_nh}}{\beta_n}\left\{ \widetilde{\bs\Phi}_n^-\mathbb{L}_n^--\widetilde{\bs\Phi}_n^+\mathbb{L}_n^+ \right\}\\
&=\frac{1}{2\mathbf{i}}\sum_{n\in\mathcal{C}_\delta}\Big\{ {\pa}_{\beta_n}\widetilde{\bs\Phi}_n^-|_{\beta_n=0}\mathbb{L}_n^-|_{\beta_n=0}+\widetilde{\bs\Phi}_n^-|_{\beta_n=0}{\pa}_{\beta_n}\mathbb{L}_n^-|_{\beta_n=0}\\
&\quad-{\pa}_{\beta_n}\widetilde{\bs \Phi}_n^+|_{\beta_n=0}\mathbb{L}_n^+|_{\beta_n=0}-\widetilde{\bs\Phi}_n^+|_{\beta_n=0}{\pa}_{\beta_n}\mathbb{L}_n^+|_{\beta_n=0} \Big\}.
\end{aligned}
\enn
Therefore, the modified finite rank operators $\widetilde{\mathbb{M}}$ can be expressed as
\ben
\mathbb{\widetilde{M}}:=\sum_{n\in\mathcal{C}_\delta\setminus\mathcal{N}}\frac{\e^{\mathbf{i}\beta_n h}}{2\mathbf{i}\beta_n}\left\{ \widetilde{\bs\Phi}_n^-\mathbb{L}_n^--\widetilde{\bs\Phi}_n^+\mathbb{L}_n^+ \right\}+\frac{1}{2\mathbf{i}}\sum_{n\in\mathcal{N}}\left\{ \widetilde{\bs\Phi}_n{\pa}_{\beta_n}(\mathbb{L}_n^--\mathbb{L}_n^+)+{\pa}_{\beta_n}\widetilde{\bs\Phi}_n(\mathbb{L}_n^-+\mathbb{L}_n^+) \right\},
\enn
where $\widetilde{\bs \Phi}_n:=\widetilde{\bs\Phi}_n^\pm$ and
\ben
{\pa}_{\beta_n}\widetilde{\bs\Phi}_n^-=-{\pa}_{\beta_n}\widetilde{\bs\Phi}_n^+={\pa}_{\beta_n}\widetilde{\bs\Phi}_n=\begin{bmatrix}
   \mathbf{i}\widetilde{x}_2\\ \widetilde{\nu}_x\cdot(-\widetilde{x}_2\alpha_n,\mathbf{i})\\ 0\\ 0 
\end{bmatrix}\e^{\mathbf{i}\alpha_nx_1}
\enn
for $n\in\mathcal{N}$. The $\beta_n$-derivative of the functionals $\mathbb{L}_n^\pm$ can be computed using the following expression
\ben
\begin{aligned}
\pa_{\beta_n}\mathbb{L}_n^\pm[\widetilde{\bs\phi}^\modi](x)&=\mp\frac{\mathbf{i}}{\Lambda}\int_{-\frac{\Lambda}{2}}^{\frac{\Lambda}{2}}\bigg\{ \widetilde{\mathscr{D}}_1^{\widetilde{\Gamma}_1^b}[\widetilde{\phi}_1^\modi]-\eta\widetilde{\mathscr{S}}_1^{\widetilde{\Gamma}_1^b}[\widetilde{\phi}_2^\modi]+\left(\zeta^{-1}\widetilde{\mathscr{D}}_1^{\Gamma_{(2,-1)}^b}-\zeta^2\widetilde{\mathscr{D}}_1^{\Gamma_{(3,1)}^b}\right)[\widetilde{\phi}_3^\modi]\\
&\quad-\left(\zeta^{-1}\widetilde{\mathscr{S}}_1^{\Gamma_{(2,-1)}^b}-\zeta^2\widetilde{\mathscr{S}}_1^{\Gamma_{(3,1)}^b}\right)[\widetilde{\phi}_4^\modi] \bigg\}(x)\e^{-\mathbf{i}\alpha_n x_1}\mathrm{d}x_1.
\end{aligned}
\enn
Similarly, the expression for the modified approximate scattered field $\widetilde{u}^\sct$ reads:
\be\label{Final modified PML approximation of u^sct including RW-anomalies}
\begin{aligned}
\widetilde{u}^{\sct}_{\modi}(x)&\approx\widetilde{\mathscr{D}}_1^{\Gamma_1^b}[\widetilde{\phi}_1^\modi](x)-\eta\widetilde{\mathscr{S}}_1^{\Gamma_1^b}[\widetilde{\phi}_2^\modi](x)\\
&\quad+\left(\widetilde{\mathscr{D}}_1^{\Gamma_2^b}-\zeta\widetilde{\mathscr{D}}_1^{\Gamma_3^b}\right)[\widetilde{\phi}_3^\modi](x)-\left(\widetilde{\mathscr{S}}_1^{\Gamma_2^b}-\zeta\mathscr{S}_1^{\Gamma_3^b}\right)[\widetilde{\phi}_4^\modi](x)\\
&\quad+\frac{1}{2\mathbf{i}}\sum_{n\in\mathcal{C}_\delta\setminus\mathcal{N}}\frac{\e^{i\beta_nh}}{\beta_n}\left\{ \widetilde{u}_n^-\mathbb{L}_n^+[\widetilde{\bs\phi}^\modi]-\widetilde{u}_n^+\mathbb{L}_n^-[\widetilde{\bs\phi}^\modi] \right\}(x)\\
&\quad+\frac{1}{2\mathbf{i}}\sum_{n\in\mathcal{N}}\pa_{\beta_n}\left\{ \widetilde{u}_n^-\mathbb{L}_n^+[\widetilde{\phi}^\modi]-\widetilde{u}_n^+\mathbb{L}_n^-[\widetilde{\phi}^\modi] \right\}(x),\quad x\in\Omega_1^b.
\end{aligned}
\en

\section{Numerical experiments.}
\label{sec:5}

In this section, we present several numerical examples to demonstrate the efficiency and accuracy of the proposed PML-BIE method for solving the problem of scattering by periodic arrays of obstacles. The BIE is numerically evaluated through the Chebyshev-based solver developed in~\cite{LXYZ23}. The resulting linear system is solved using the fully complex version of the iterative solver GMRES with a residual tolerance of $\epsilon_r=10^{-14}$. To illustrate the performance of the method, we compute the energy balance, self-convergence, quasi-periodicity mismatch and radiation condition errors for the numerical solution. All numerical results presented in this paper were obtained using Fortran implementations of the algorithms parallelized with OpenMP.

Since the exact solution satisfies \eqref{energy balance}, the numerical errors can be assessed by evaluating the following energy balance error
\be\label{energy balance error}
\mathrm{error}_{eb}:=\bigg| 2\real(B_0^-)+\sum_{n\in\mathcal{P}}\frac{\beta_n}{\beta}\big\{ |B_n^-|^2+|B_n^+|^2 \big\} \bigg|,
\en
where the coefficients in \eqref{energy balance error} can be calculated via 
\be
B_n^\pm:=\frac{\e^{-\mathbf{i}\beta_n \widetilde{h}}}{\Lambda}\int_{-\frac{\Lambda}{2}}^{\frac{\Lambda}{2}}\widetilde{u}^{\sct}(x_1,\pm \widetilde{h})\e^{-\mathbf{i}\alpha_n x_1}\mathrm{d}x_1,\quad n\in\mathcal{P},
\en
where $\max\{h^+, -h^-\} < \widetilde{h} < H$. We also evaluate the self-convergence error, also called relative error, which is defined as
\be\label{self-convergence error}
\mathrm{error}_{sc} = \frac{\max_{p=1,\ldots,4}|u_{\mathrm{num}}(\widetilde{x}^p) - u_{\mathrm{ref}}(\widetilde{x}^p)|}{\max_{p=1,\ldots,4}|u_{\mathrm{ref}}(\widetilde{x}^p)|}.
\en
where $u_\mathrm{num}$ is the numerical solution of $\widetilde{u}^{\sct}$ and $u_{\mathrm{ref}}$ is produced by means of numerical solution with sufficiently fine discretizations, and the sample points are set to be $\widetilde{x}^1=(-0.5,-0.5)$, $\widetilde{x}^2=(-0.5, 0.5)$, $\widetilde{x}^3=(0.5,-0.5)$, $\widetilde{x}^4=(0.5,0.5)$. To verify the quasi-periodicity condition, we then introduce the right and left quasi-periodicity mismatch errors 
\be\label{quasi-periodicity mismatch error}
\begin{aligned}
&\mathrm{error}_{qp}^{(r)}:=\frac{\max_{p=1,\ldots,4}|\widetilde{u}^{\sct}(\widetilde{x}^p)-\zeta^{-1}\widetilde{u}^{\sct}(\widetilde{x}^p+\Lambda(1,0))|}{\max_{p=1,\ldots,4}|\widetilde{u}^{\sct}(\widetilde{x}^p)|},\\
&\mathrm{error}_{qp}^{(l)}:=\frac{\max_{p=1,\ldots,4}|\widetilde{u}^{\sct}(\widetilde{x}^p)-\zeta\widetilde{u}^{\sct}(\widetilde{x}^p-\Lambda(1,0))|}{\max_{p=1,\ldots,4}|\widetilde{u}^{\sct}(\widetilde{x}^p)|},
\end{aligned}
\en
respectively, where the sample points are the same as the self-convergence error. To examine the enforcement of the radiation condition, the radiation condition errors are defined as
\be\qquad\label{radiation condition error}
\begin{aligned}
\mathrm{error}_{rc}^{(\pm,n)}&:=\bigg| \frac{1}{\Lambda}\int_{-\frac{\Lambda}{2}}^{\frac{\Lambda}{2}}\big\{ \widetilde{\pa}_{x_2}\widetilde{u}^{\sct}(x_1,\pm\widetilde{h})\mp\mathbf{i}\beta_n\widetilde{u}^{\sct}(x_1,\pm\widetilde{h}) \big\}\e^{-\mathbf{i}\alpha_n x_1}\mathrm{d}x_1 \bigg|.
\end{aligned}
\en

\begin{figure}[htbp]
   \centering

    \begin{tabular}{ccccc}
    \subfloat[Energy balance]{
    \includegraphics[scale=0.21]{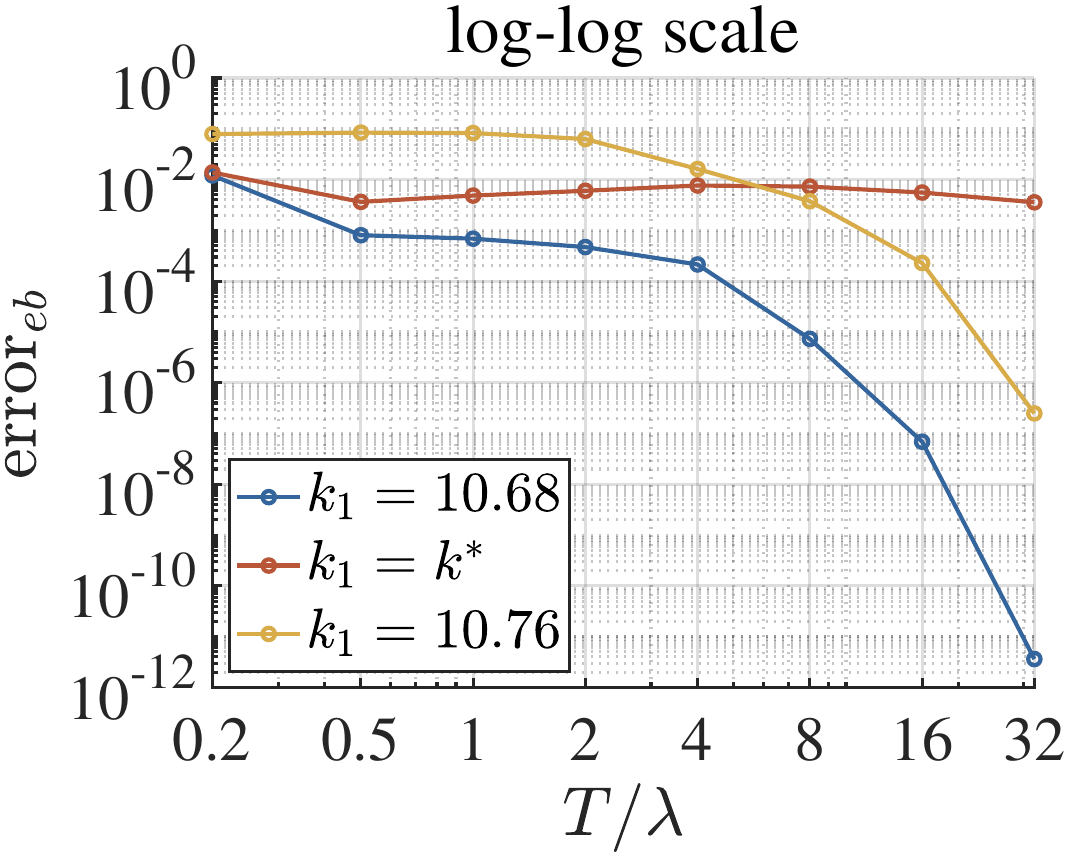}
    \label{fig:Error_eb_Noncorrect_CW03}
    }
    \subfloat[Self-convergence]{
    \includegraphics[scale=0.21]{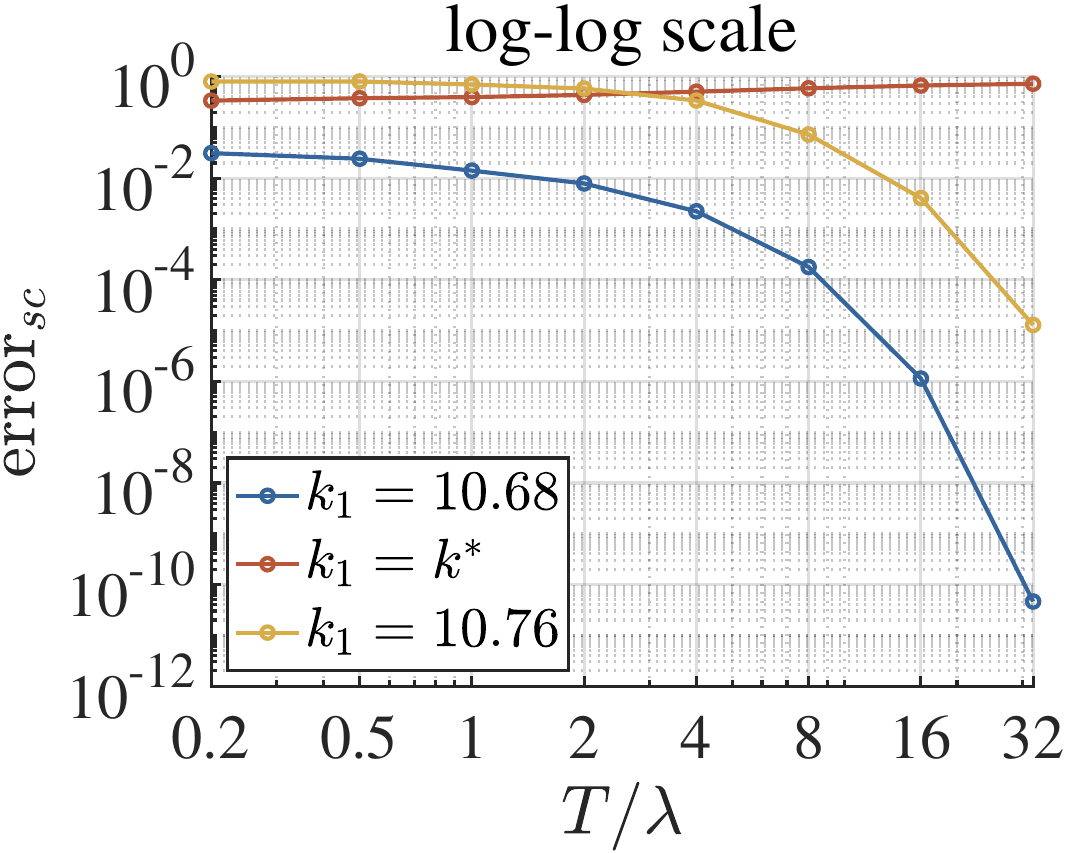}
    \label{fig:Error_sc_Noncorrect_CW03}
    }
    \subfloat[Quasi-periodicity]{
    \includegraphics[scale=0.21]{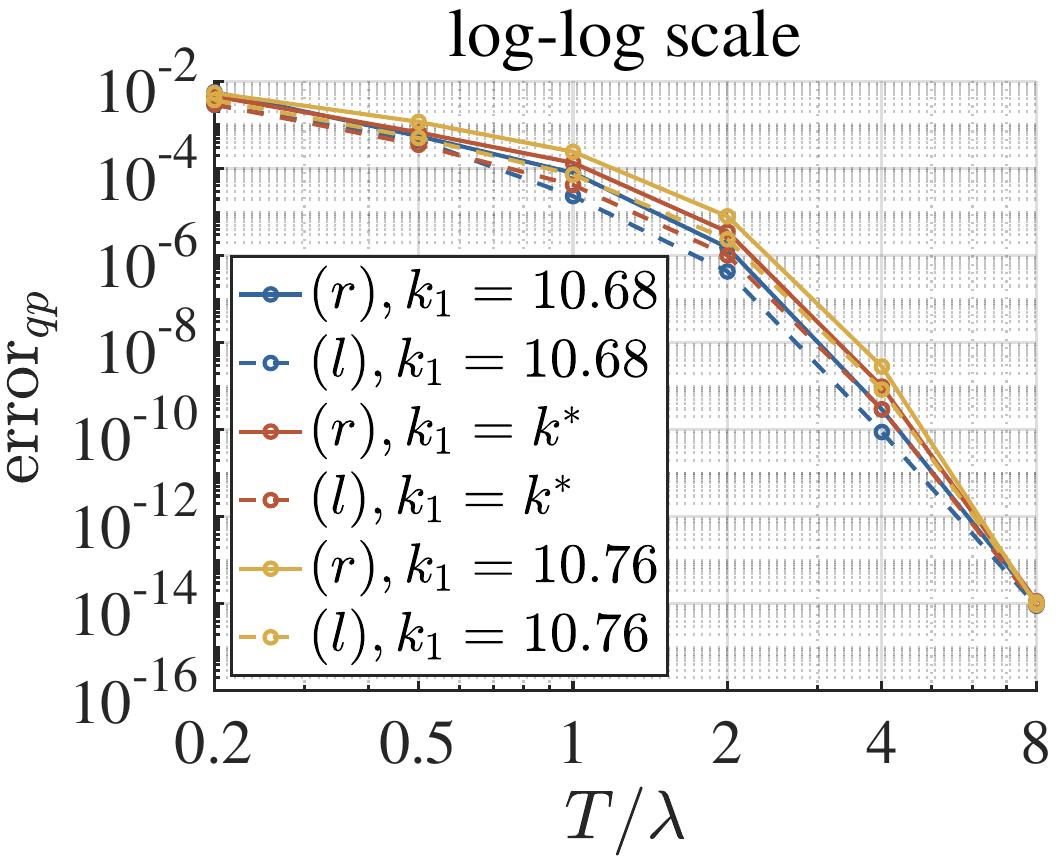}
    \label{fig:Error_qp_Noncorrect_CW03}
    }\\
    \subfloat[Energy balance]{
    \includegraphics[scale=0.232]{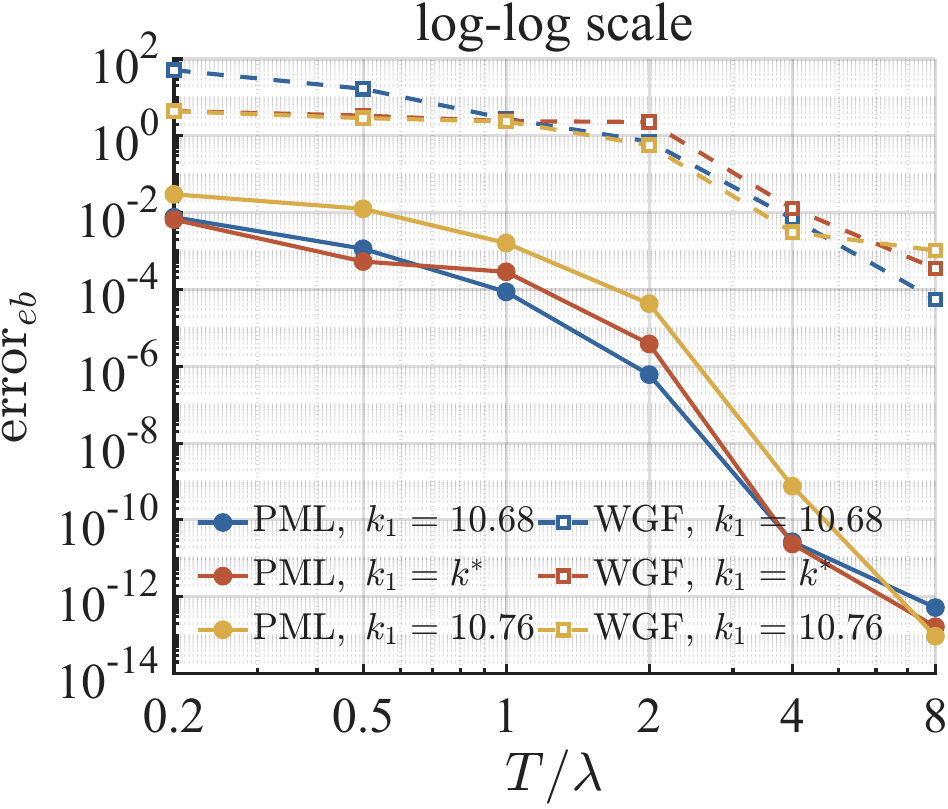}
    \label{fig:Error_eb_Correct_CW03}
    }
    \subfloat[Self-convergence]{
    \includegraphics[scale=0.232]{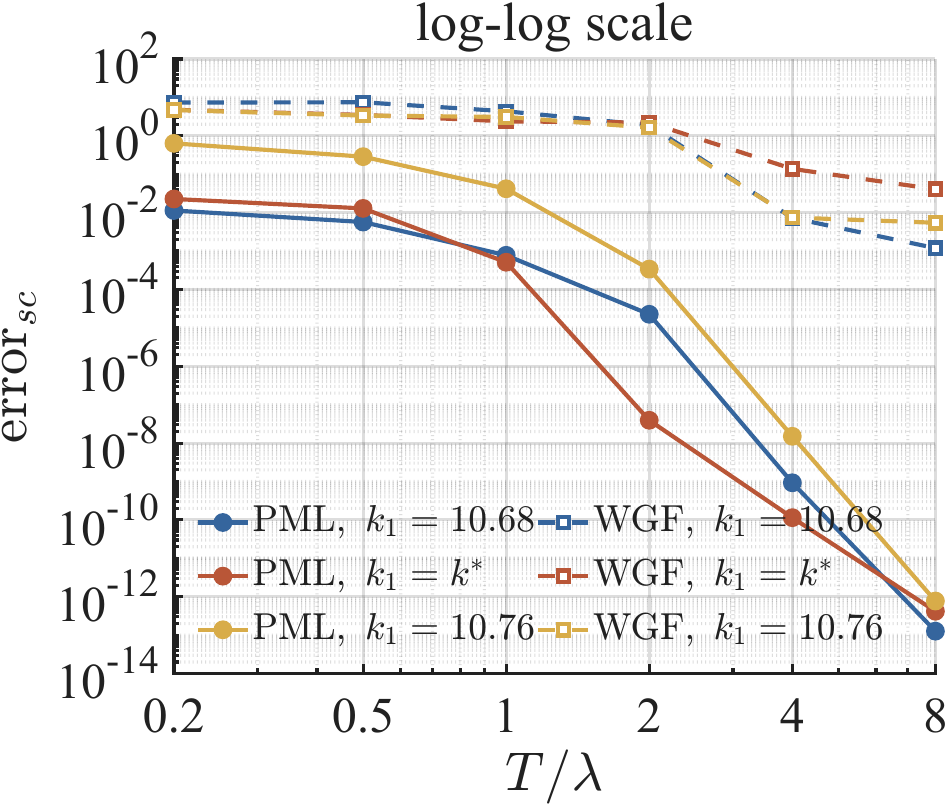}
    \label{fig:Error_sc_Correct_CW03}
    }
    \subfloat[Quasi-periodicity]{
    \includegraphics[scale=0.232]{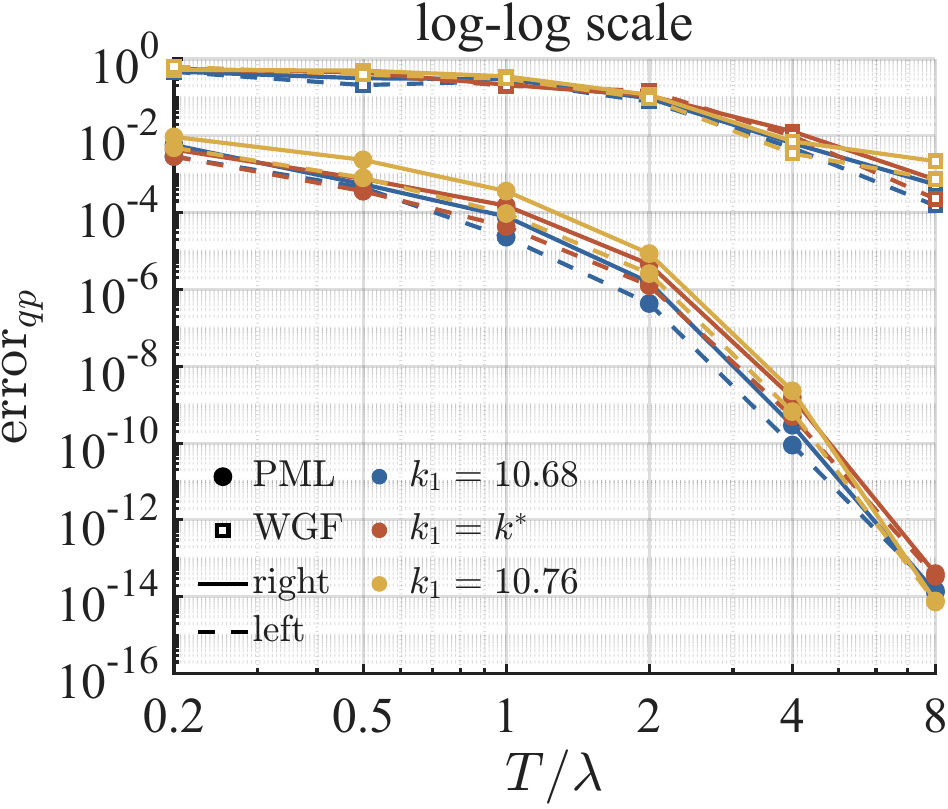}
    \label{fig:Error_qp_Correct_CW03}
    }
    \end{tabular}
\caption{Example 1. Energy-balance errors \eqref{energy balance error}, self-convergence errors \eqref{self-convergence error}, and quasi-periodicity errors \eqref{quasi-periodicity mismatch error}: preliminary truncated PML-BIE method in the upper row, and modified PML-BIE versus corrected WGF methods in the lower row.}
\label{Example1-1}
\end{figure}

\begin{figure}[htbp]
\centering

  \subfloat[$k_1=10.68$]{
  \includegraphics[scale=0.21]{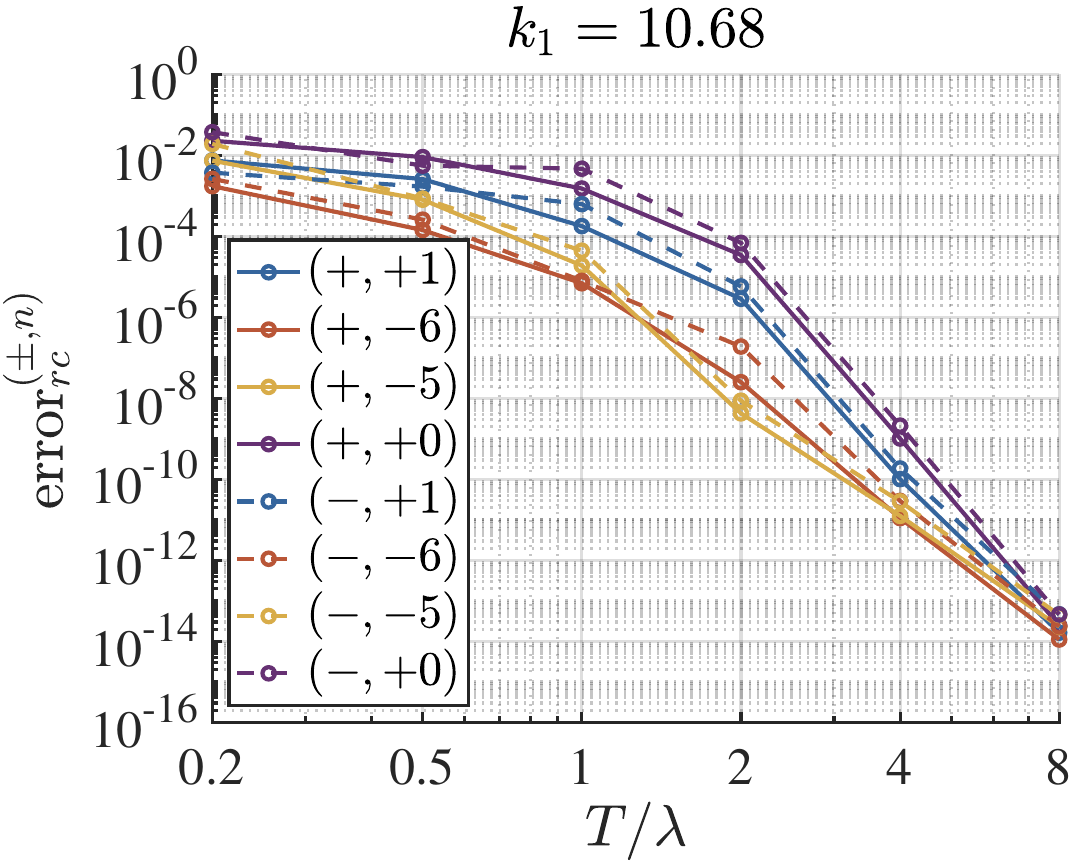}
  \label{fig:Error_rc_1068_Correct_CW03}
  }
  \subfloat[$k_1=k^*\approx10.7261$]{
  \includegraphics[scale=0.21]{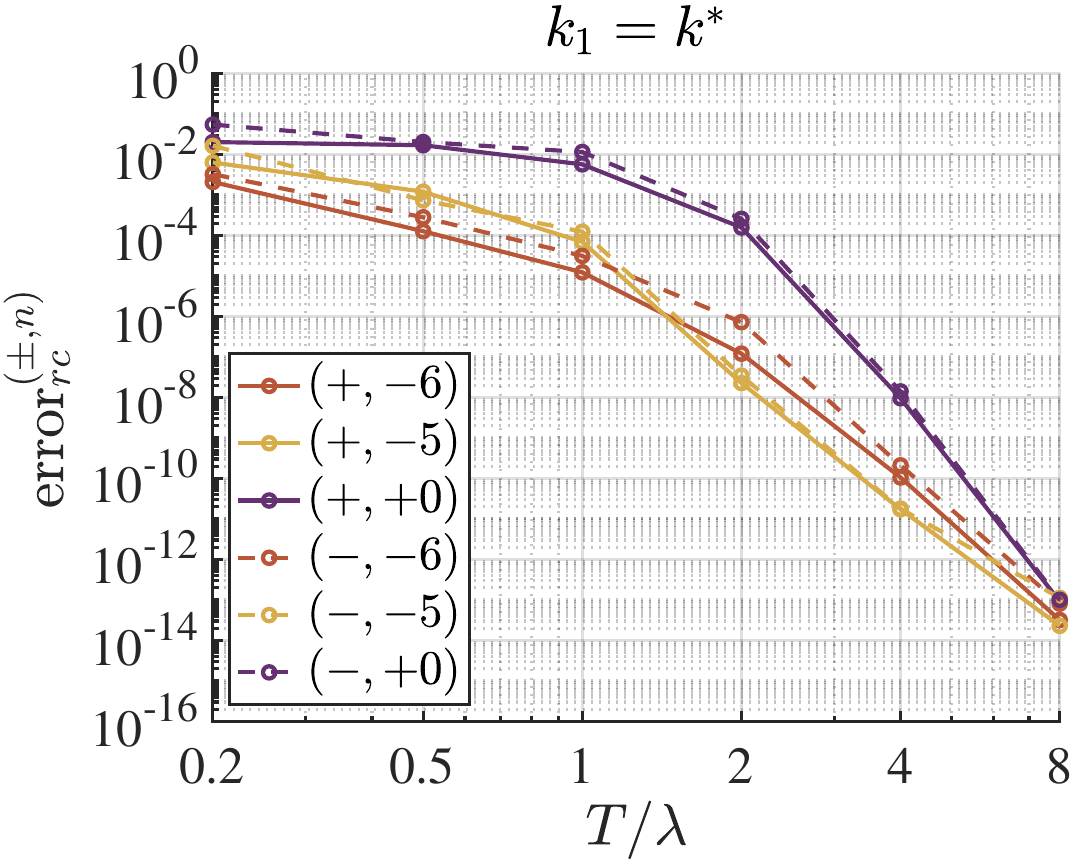}
  \label{fig:Error_rc_1072_Correct_CW03}
  }
  \subfloat[$k_1=10.76$]{
  \includegraphics[scale=0.21]{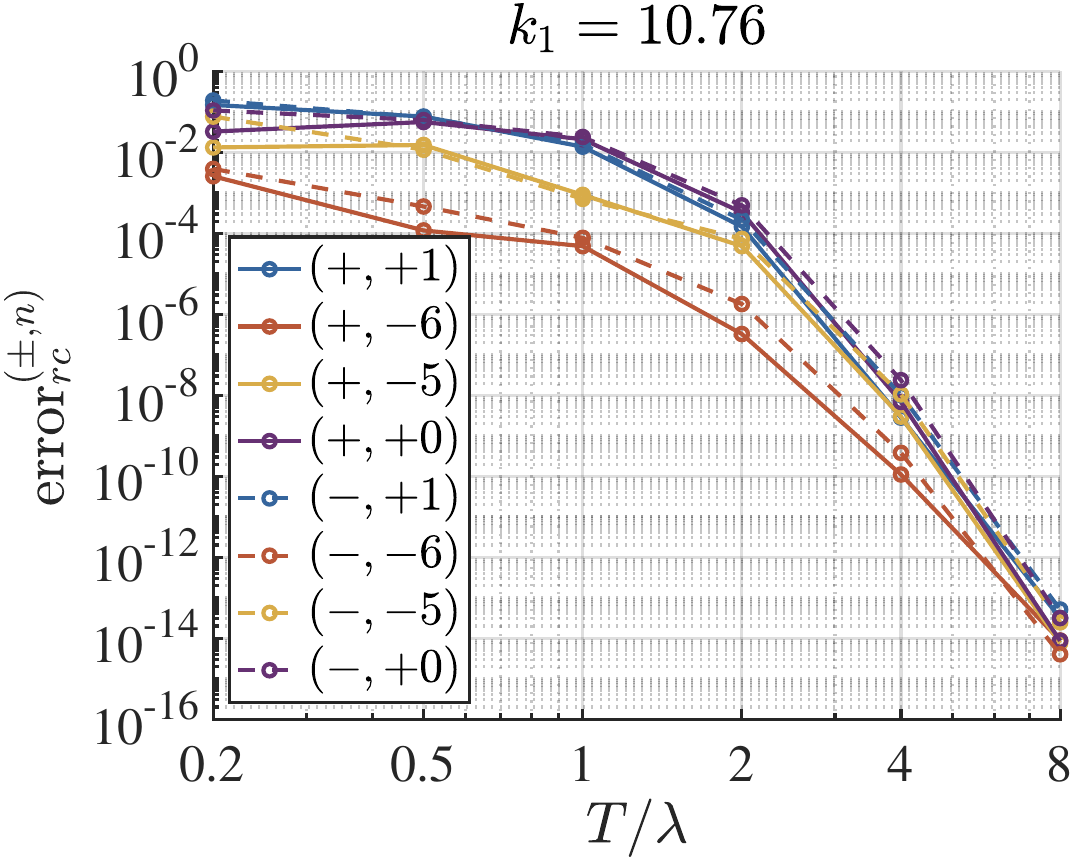}
  \label{fig:Error_rc_1076_Correct_CW03}
  }
\caption{Example 1. Radiation condition errors for the modified PML-BIE method.}
\label{Example1-2}
\end{figure}

\begin{table}[htbp]
  \centering
  \caption{Example 1. Numerical errors for the problem considered in Figure~\ref{Example1-3}. The first three rows report the modified PML-BIE results with $T=4\lambda$, while the last three rows report the corrected WGF results with window size $T=32\lambda$.}
  \label{Tab:Ex1_numerical_errors}
  \setlength{\tabcolsep}{4pt}
  \renewcommand{\arraystretch}{1.15}
  \begin{tabular}{ccccccc}
  \hline
  $k_1$ 
  & $\mathrm{error}_{eb}$ 
  & $\mathrm{error}_{sc}$ 
  & $\mathrm{error}_{qp}^{(l)}$ 
  & $\mathrm{error}_{qp}^{(r)}$ 
  & $\mathrm{error}_{rc}^{(+,-5)}$ 
  & $\mathrm{error}_{rc}^{(-,-5)}$ \\
  \hline
  \multicolumn{7}{c}{Modified PML-BIE method, $T=4\lambda$} \\
  \hline
  $10.68$ & $2.69\mathrm{E}{-11}$ & $9.25\mathrm{E}{-10}$ & $2.91\mathrm{E}{-10}$ & $8.86\mathrm{E}{-11}$ & $1.25\mathrm{E}{-11}$ & $2.84\mathrm{E}{-11}$ \\
  $k^*$   & $2.36\mathrm{E}{-11}$ & $1.14\mathrm{E}{-10}$ & $1.55\mathrm{E}{-9}$  & $4.90\mathrm{E}{-10}$ & $1.73\mathrm{E}{-11}$ & $1.78\mathrm{E}{-11}$ \\
  $10.76$ & $7.55\mathrm{E}{-10}$ & $1.50\mathrm{E}{-8}$  & $2.29\mathrm{E}{-9}$  & $6.59\mathrm{E}{-10}$ & $2.97\mathrm{E}{-9}$  & $1.05\mathrm{E}{-8}$ \\

  \hline
  \multicolumn{7}{c}{Corrected WGF method, $T=32\lambda$} \\
  \hline
  $10.68$ & $5.08\mathrm{E}{-8}$ & $1.68\mathrm{E}{-7}$ & $3.92\mathrm{E}{-9}$ & $4.37\mathrm{E}{-9}$ & $9.28\mathrm{E}{-7}$ & $1.96\mathrm{E}{-8}$ \\
  $k^*$   & $9.93\mathrm{E}{-7}$ & $4.26\mathrm{E}{-3}$ & $3.46\mathrm{E}{-8}$ & $4.80\mathrm{E}{-8}$ & $1.27\mathrm{E}{-6}$ & $3.07\mathrm{E}{-6}$ \\
  $10.76$ & $2.92\mathrm{E}{-5}$ & $2.25\mathrm{E}{-5}$ & $2.06\mathrm{E}{-7}$ & $2.84\mathrm{E}{-7}$ & $3.18\mathrm{E}{-6}$ & $1.28\mathrm{E}{-5}$ \\
  \hline
  \end{tabular}
\end{table}

\begin{figure}[htbp]
\centering
\setlength{\tabcolsep}{1pt}
\begin{tabular}{ccc}
\subfloat[$k_1=10.68$]{
  \includegraphics[width=0.30\textwidth]{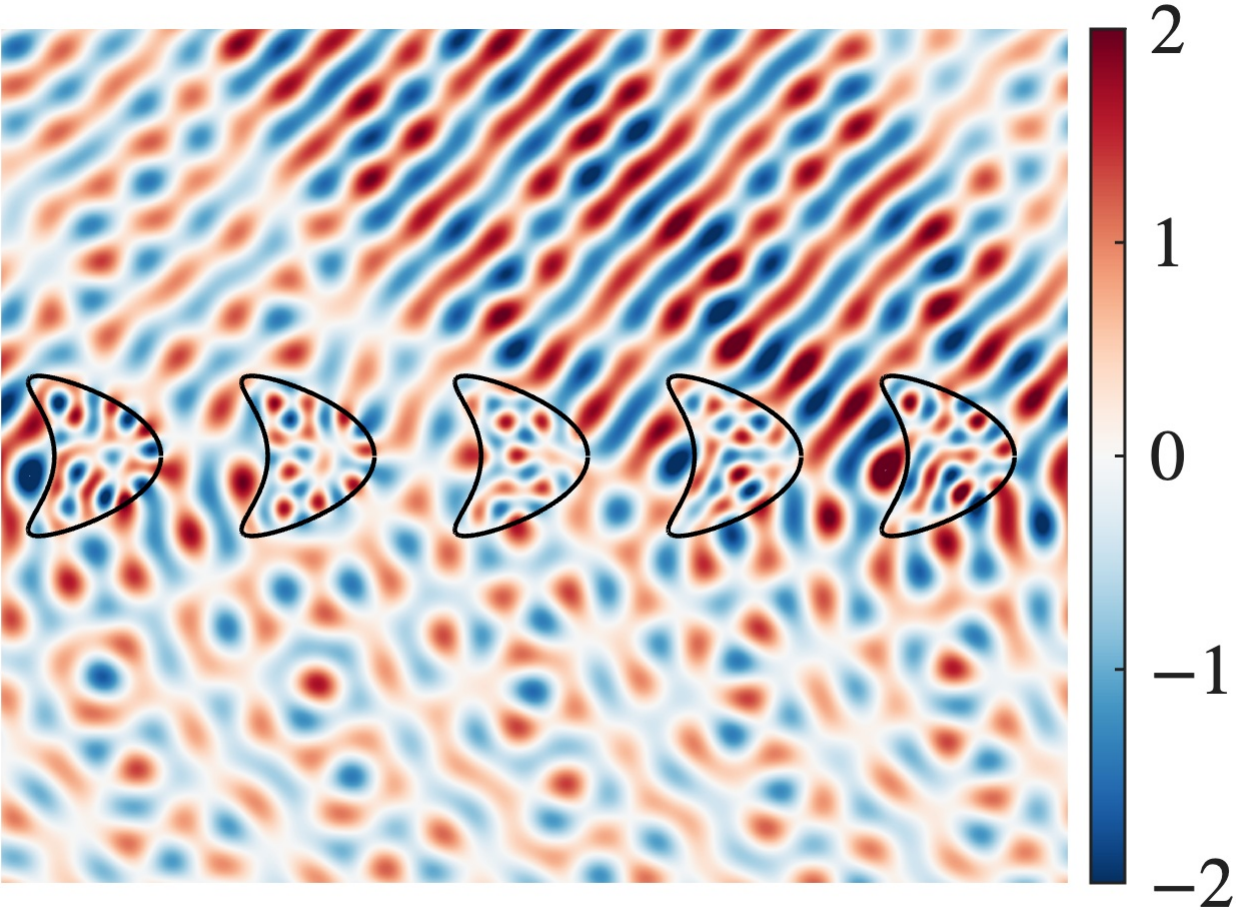}
  \label{fig:Ex1_benchmark_1068_Re}
}
&
\subfloat[$k_1=k^*$]{
  \includegraphics[width=0.30\textwidth]{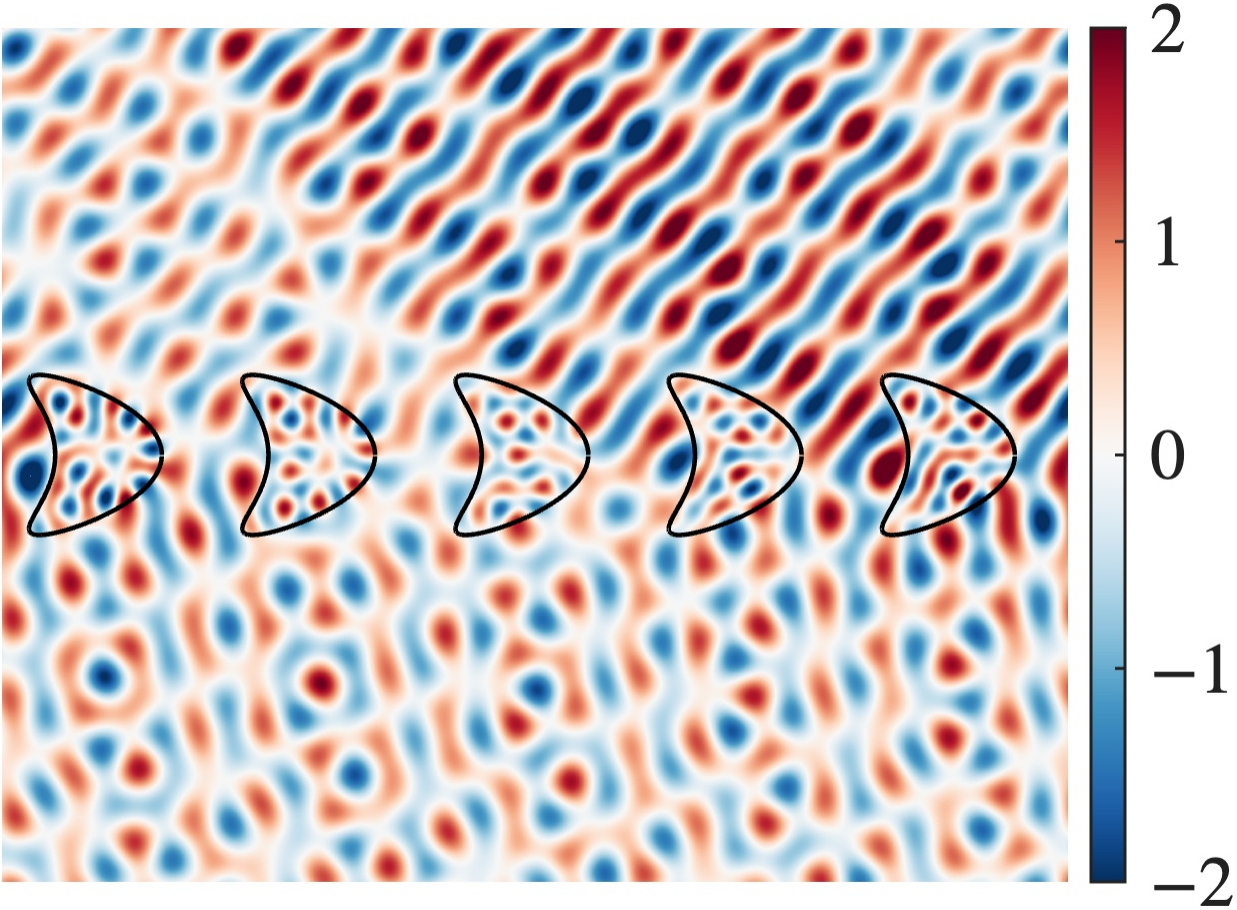}
  \label{fig:Ex1_benchmark_1072_Re}
}
&
\subfloat[$k_1=10.76$]{
  \includegraphics[width=0.30\textwidth]{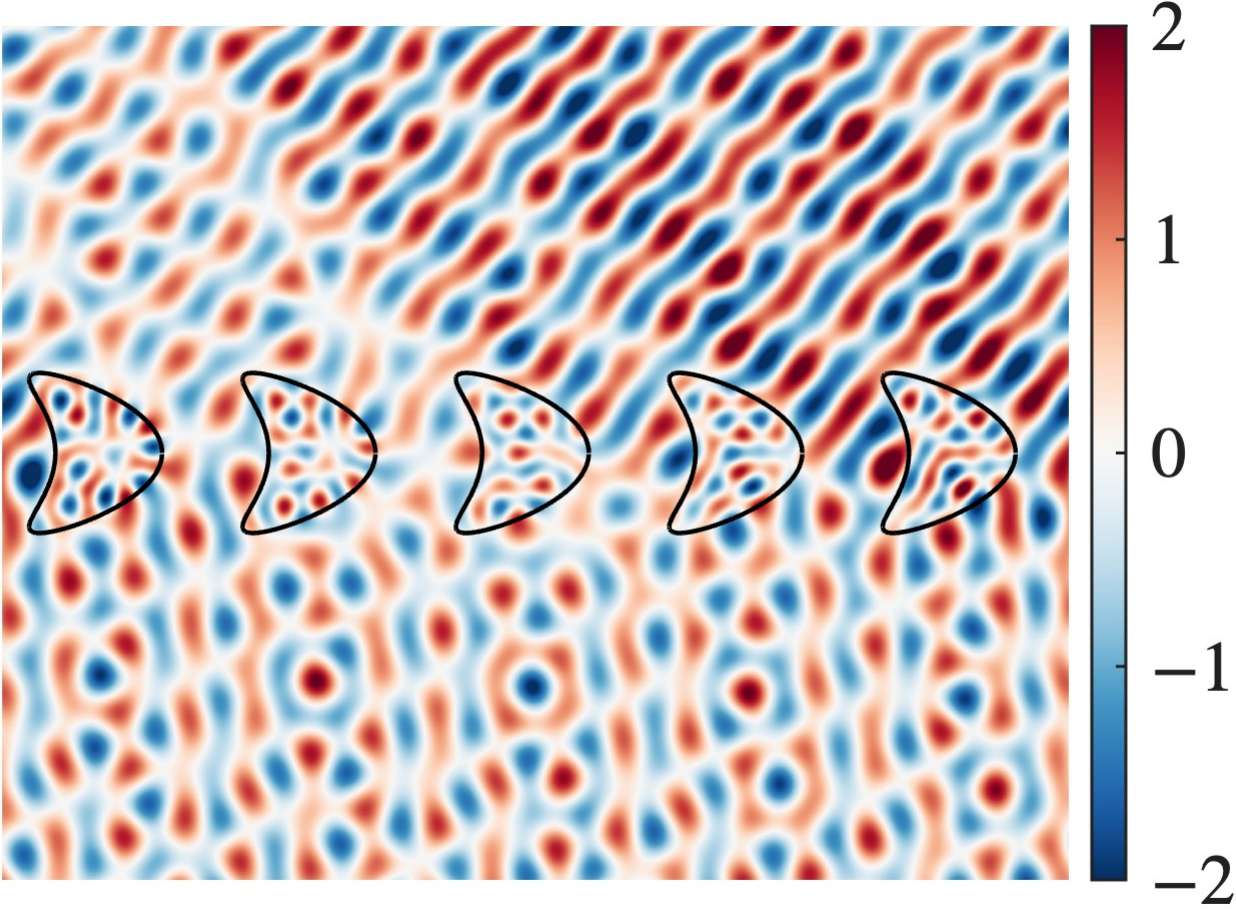}
  \label{fig:Ex1_benchmark_1076_Re}
}
\end{tabular}
\caption{Example 1. Real parts of the total fields resulted from the modified PML-BIE method.}
\label{Example1-3}
\end{figure}

{\bf Example 1.}
We consider the diffraction and transmission of a TE-polarized plane wave impinging at an angle $\theta^\inc=\pi/4$ on periodic arrays of period $\Lambda=2$ consisting of kite-shaped obstacles parameterized by
\ben
\bs r_1(t)=r(\cos t+0.65\cos 2t-0.65,1.5\sin t),\quad t\in I_{2\pi},
\enn
with $r=0.5$ throughout this example. We choose the parameter values $k_2=20$, $\widetilde{h}=1$, $S=6$, $H=4$, $P=8$, and $\eta=1$. The first row of Figure~\ref{Example1-1} reports the energy-balance errors \eqref{energy balance error}, self-convergence errors \eqref{self-convergence error}, and quasi-periodicity mismatch errors \eqref{quasi-periodicity mismatch error} for the preliminary truncated PML-BIE method. These errors decay exponentially with respect to the PML thickness $T$ away from the RW-anomaly, but the convergence deteriorates near $k_1=k^*=2\pi/(\Lambda(1-\sin\theta^\inc))\approx 10.7261$, showing the sensitivity of the preliminary truncated formulation at the RW-anomaly. {The second row of Figure~\ref{Example1-1} presents a comparison of the proposed PML-BIE method with the corrected WGF method developed in \cite{SFFP23}. It shows that the PML-BIE method exhibits exponential convergence with respect to $T$ (in particular, $T$ also denotes the window size of WGF method) which takes advantages over the super-algebraically convergence enjoyed by the corrected WGF method and therefore, the proposed new method can attain higher accuracy with a much smaller truncation thickness $T$, see Table~\ref{Tab:Ex1_numerical_errors} for a detailed comparison of the numerical errors of the PML-BIE method with $T=4\lambda$ and the WGF method with window size $T=32\lambda$.}
The radiation condition errors \eqref{radiation condition error} in Figure~\ref{Example1-2}, computed for $k_1=10.68$, $k_1=k^*$, and $k_1=10.76$, also decay exponentially as $T$ increases, confirming that the modified PML-BIE method accurately captures the outgoing radiation behavior near the RW-anomaly. 

\begin{figure}[htbp]
   \centering
   \subfloat[$k_1=10.70$]{
   \includegraphics[scale=0.30]{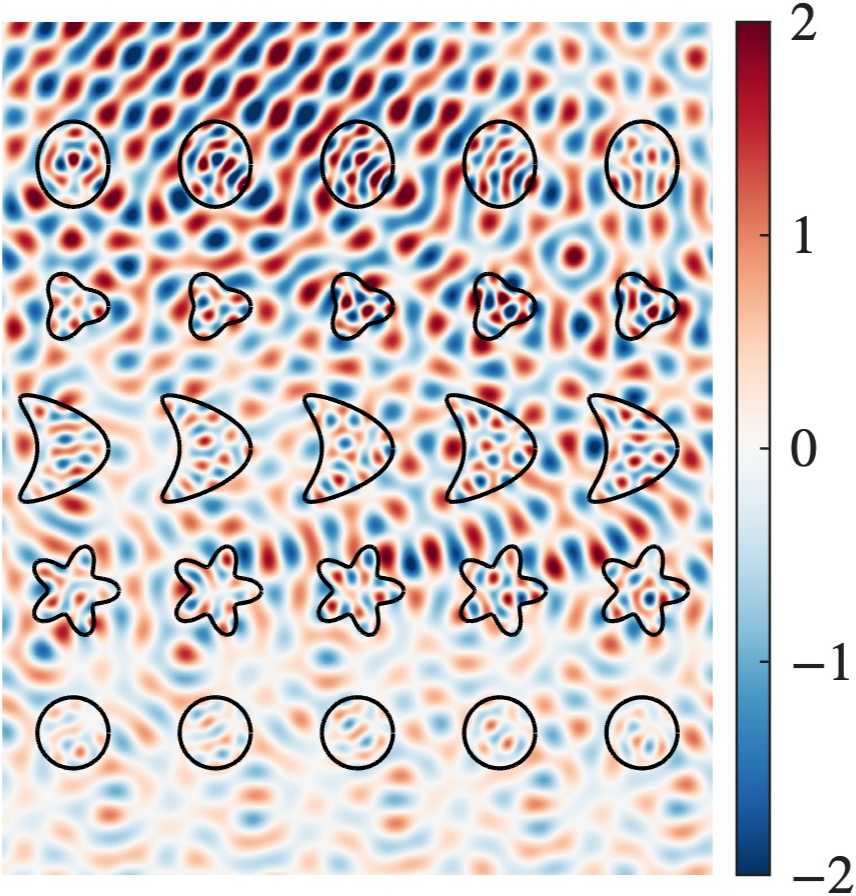}
   \label{fig:Ex2-1-1070-Re}
   }\qquad
   \subfloat[$k_1=10.78$]{
   \includegraphics[scale=0.30]{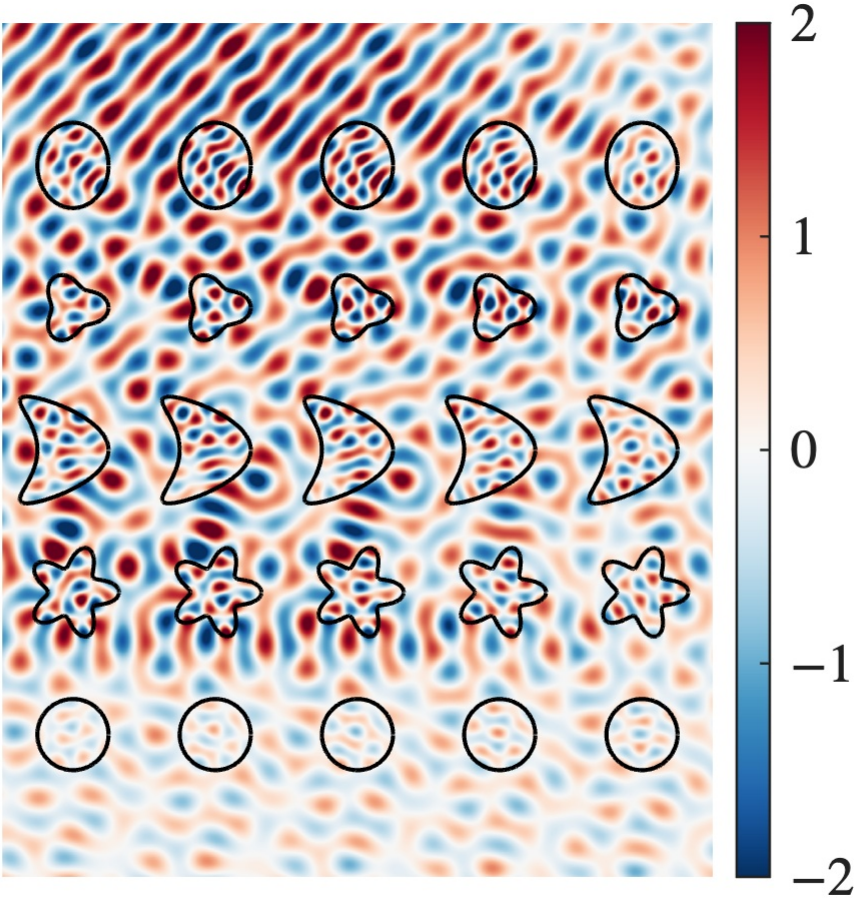}
   \label{Ex2-1-1078-Re}
   }
\caption{Example 2. Real part of the total field for obstacles of varying shapes resulting from the modified PML-BIE method.}
\label{Example2-1}
\end{figure}

   \begin{table}[htbp]
    \caption{Example 2. Numerical errors for the problem considered in Figures \ref{fig:Ex2-1-1070-Re} and \ref{Ex2-1-1078-Re} with $T=4\lambda$.}
    \centering
    \begin{tabular}{ccccccccc}
    \hline
    $k_1$ & $\mathrm{error}_{eb}$ & $\mathrm{error}_{sc}$ & $\mathrm{error}_{qp}^{(l)}$ & $\mathrm{error}_{qp}^{(r)}$ & $\mathrm{error_{rc}^{(+,+1)}}$ & $\mathrm{error}_{rc}^{(-,+1)}$  \\ \hline
    10.70& 3.86E-11& 3.32E-11& 3.29E-11& 1.55E-11& 3.62E-11& 5.60E-11\\ 
    10.78& 1.14E-10& 8.41E-11& 7.05E-11& 2.20E-11& 9.86E-11 & 6.63E-11\\\hline
    \end{tabular}
    \label{Tab:Ex3_numerical_errors}
  \end{table}

\begin{figure}[htbp]
   \centering
   \subfloat[Energy balance]{
   \includegraphics[scale=0.22]{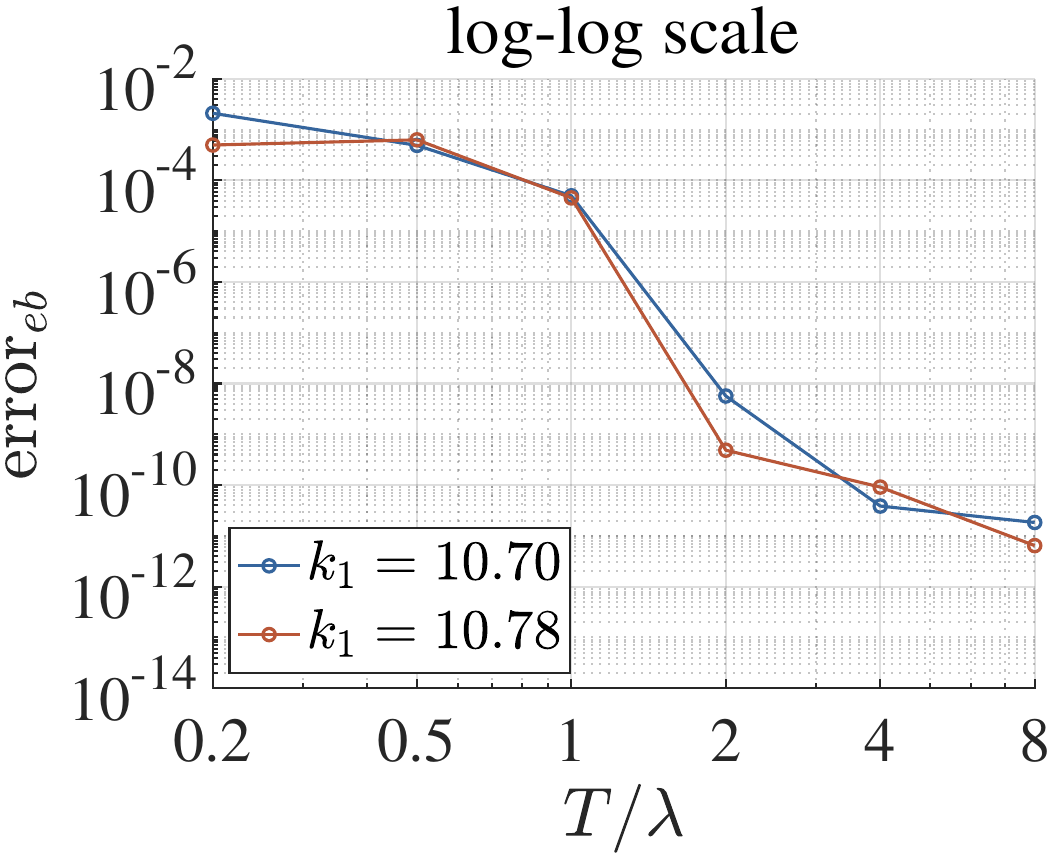}
   \label{fig:Ex3_eb_error}
   }
   \subfloat[Self-convergence]{
   \includegraphics[scale=0.22]{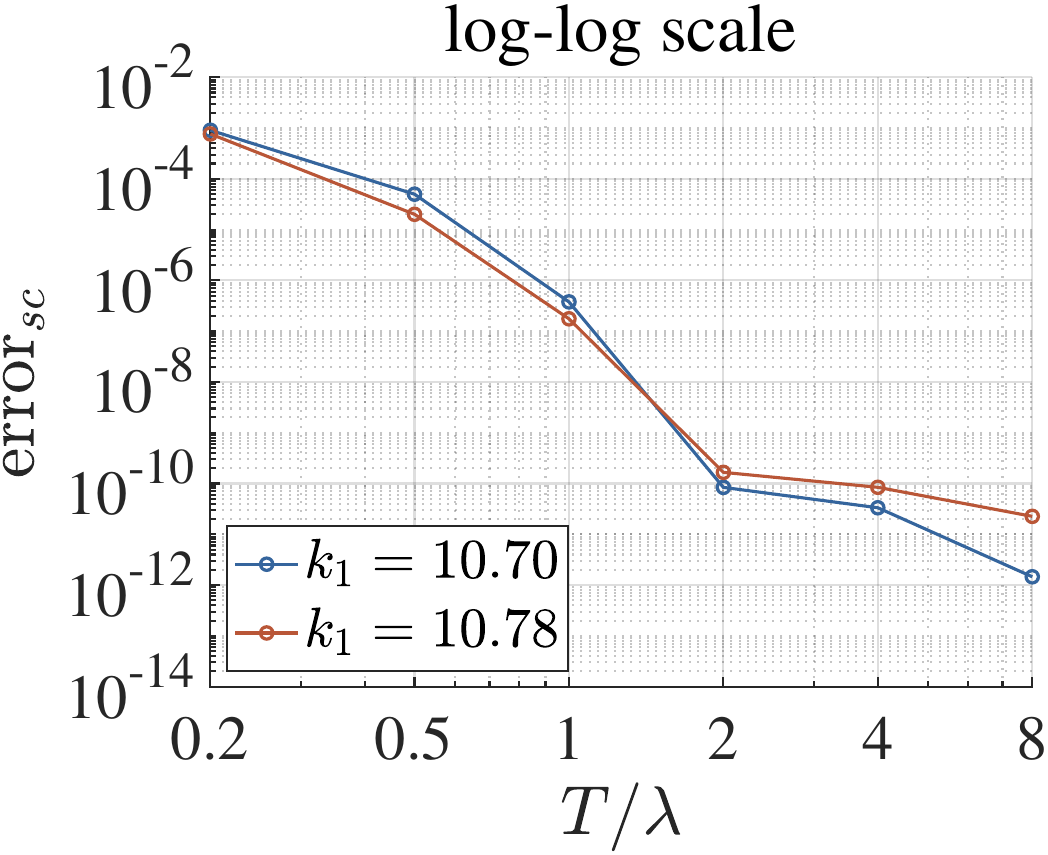}
   \label{fig:Ex3_sc_error}
   }
   \subfloat[Quasi-periodicity]{
   \includegraphics[scale=0.22]{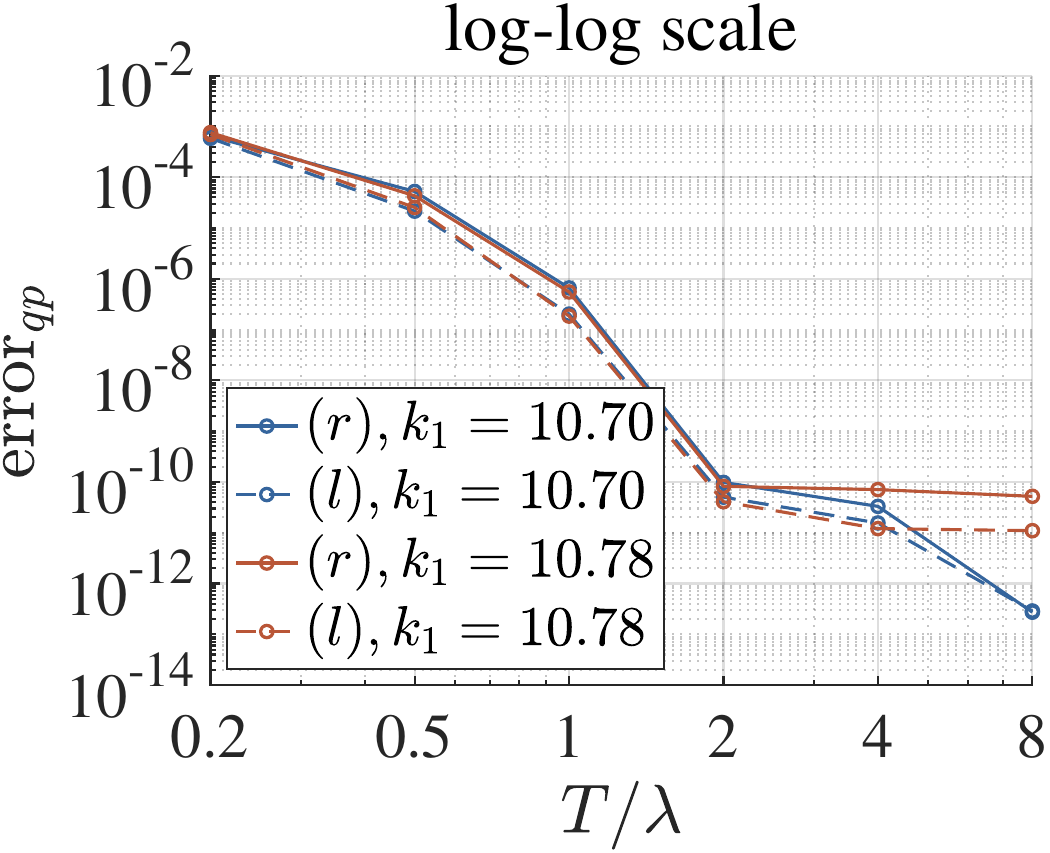}
   \label{fig:Ex3_qp_error}
   }
\caption{Example 2. Energy-balance \eqref{energy balance error}, self-convergence \eqref{self-convergence error}, and quasi-periodicity errors \eqref{quasi-periodicity mismatch error} for the modified PML-BIE methods.}
\label{Example2-2}
\end{figure}

\begin{figure}[htbp]
\centering
  \subfloat[$k_1=10.70$]{
  \includegraphics[scale=0.30]{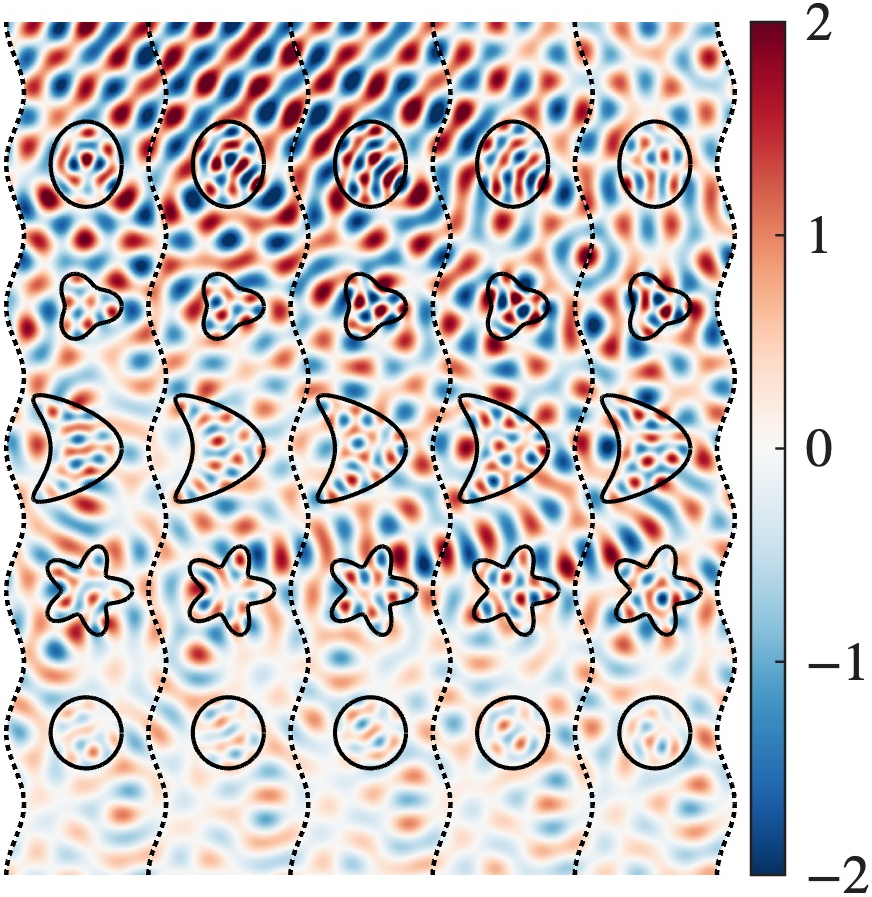}
  \label{Ex2-2-1070-Re}
  }\qquad
  \subfloat[$k_1=10.78$]{
  \includegraphics[scale=0.30]{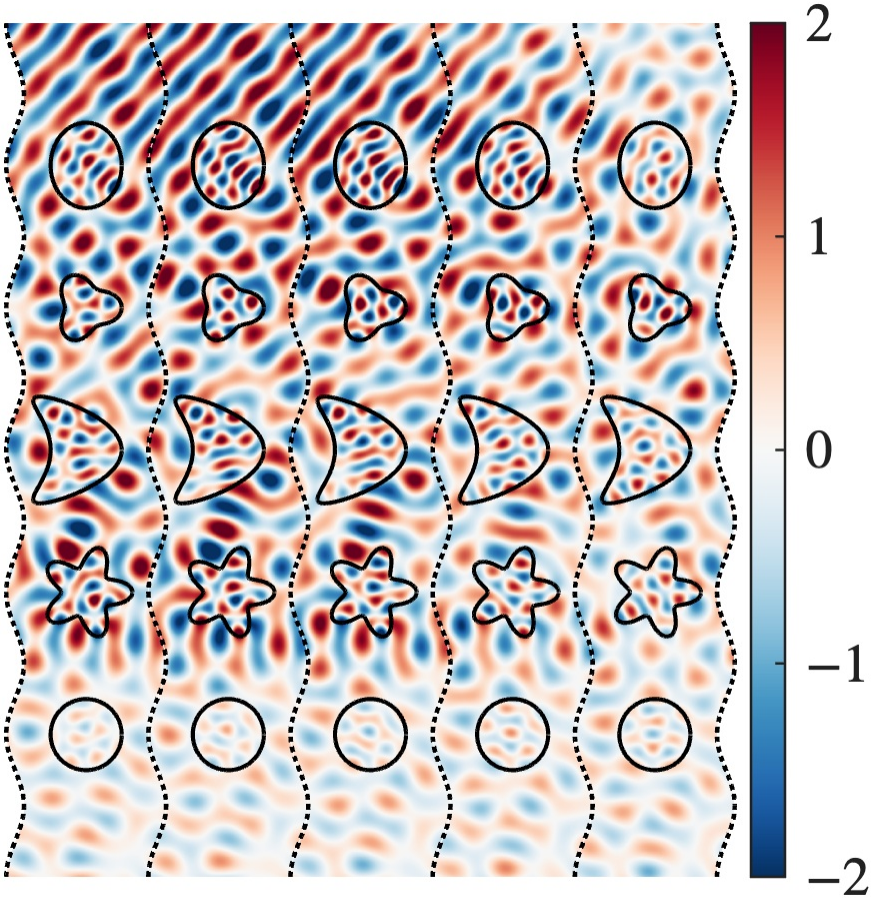}
  \label{Ex2-2-1078-Re}
  }
\caption{Example 2. Real part of the total field for obstacles of varying shapes with non-flat boundaries on $\Gamma_2$ resulting from the modified PML-BIE method.}
\label{Example2-4}
\end{figure}

\begin{table}[htbp]
  \caption{Example 2. Numerical errors for the problem considered in Figures \ref{Ex2-2-1070-Re} and \ref{Ex2-2-1078-Re} with $T=4\lambda$.}
  \centering
  \begin{tabular}{ccccccccc}
  \hline
  $k_1$ & $\mathrm{error}_{eb}$ & $\mathrm{error}_{sc}$ & $\mathrm{error}_{qp}^{(l)}$ & $\mathrm{error}_{qp}^{(r)}$ & $\mathrm{error_{rc}^{(+,+1)}}$ & $\mathrm{error}_{rc}^{(-,+1)}$  \\ \hline
  10.70& 9.34E-11& 2.99E-11& 6.75E-11& 7.18E-11 & 6.74E-13& 8.32E-13\\ 
  10.78& 2.30E-10& 1.39E-10& 8.70E-11& 6.76E-11& 6.91E-13 & 7.93E-13\\\hline
  \end{tabular}
  \label{Tab:Ex4_numerical_errors}
\end{table}

{\bf Example 2.} We consider the scattering problem involving five layered periodic obstacles of different types: ellipse, rounded triangle, kite, rounded star, and circle. The incident wave and PML parameters are the same as in Example 1, except that $H=6$ and $h=5.4$. Figures \ref{fig:Ex2-1-1070-Re} and \ref{Ex2-1-1078-Re} show the real part of the total field for $k_1=10.70$ and $k_1=10.78$, respectively, with $T=4\lambda$. The numerical errors for Figures \ref{fig:Ex2-1-1070-Re} and \ref{Ex2-1-1078-Re} are listed in Table \ref{Tab:Ex3_numerical_errors}, which demonstrate the high accuracy of the proposed method. Figure~\ref{Example2-2} illustrates that the energy-balance, self-convergence, and quasi-periodicity mismatch errors decay exponentially as the PML thickness $T$ increases, even at RW-anomaly $k_1=k^*$. Figures \ref{Ex2-2-1070-Re} and \ref{Ex2-2-1078-Re} exhibit the numerical total field for the same scattering problem as in Figures \ref{fig:Ex2-1-1070-Re} and \ref{Ex2-1-1078-Re}, but with non-straight artificial boundaries $\Gamma_2$ and $\Gamma_3$ (dotted lines). Note that the non-straight boundaries $\Gamma_2$ and $\Gamma_3$ are designed to avoid intersection with the pores $\Gamma_1$. The numerical errors for Figure \ref{Example2-4} are listed in Table \ref{Tab:Ex4_numerical_errors}, which again verifies the high accuracy of the proposed method.

\section{Conclusion.}
\label{sec:6}

This paper proposed a new frequency-robust PML-BIE method for solving two-dimensional problem of scattering by infinite periodic arrays of penetrable obstacles with plane incidence. For the preliminary truncated PML-BIE formulation, we proved the exponential convergence of its propagation/non-propagation tail term for both $n\in\mathcal{P}$ and $n\in\mathcal{Q}$, and analyze the cause of its break down near/at RW-anomalies. To overcome these challenges, we subsequently propose a modified PML-BIE solver that combines the complex-stretched PML truncation with a finite-rank operator correction. The extension of the proposed high-order, frequency-robust PML-BIE solver to three-dimensional Helmholtz scattering problems involving bi-periodic arrays of obstacles is currently underway. The method can also be extended to elastic and electromagnetic wave scattering problems. These extensions are beyond the scope of the present paper and will be pursued in future work.

\section*{Acknowledgments}
This work is partially supported by the National Key R\&D Program of China (No. 2024YFA1016000), the Strategic Priority Research Program of the Chinese Academy of Sciences through Grant No. XDB0640000 and NSFC through Grants No. 12171465, 12271082, 62231016.

\bibliographystyle{siamplain}
\bibliography{references}

\end{document}